\documentclass[reqno]{amsart}
\usepackage[utf8]{inputenc}
\usepackage{amssymb, amsmath, amstext, amsopn, amsthm, amscd}

\usepackage[numbers]{natbib}

\usepackage{tikz}
\usepackage{graphicx}
\usetikzlibrary{positioning}
\usepackage{tikz-cd}
\usepackage{hyperref}

\usepackage{algpseudocode}
\usepackage{algorithm}
\usepackage{todonotes}
\usepackage{epstopdf}

\newcommand{\Cc}{\mathbb{C}}
\newcommand{\Rr}{\mathbb{R}}

\newcommand{\Ss}{\mathbb{S}}
\newcommand{\GG}{\mathfrak{g}}
\newcommand{\NN}{\mathfrak{n}}
\newcommand{\GL}{\mathfrak{gl}}
\newcommand{\SL}{\mathfrak{sl}}
\newcommand{\SU}{\mathfrak{su}}
\providecommand{\U}{\mathfrak{u}}
\renewcommand{\U}{\mathfrak{u}}
\newcommand{\SO}{\mathfrak{so}}
\newcommand{\SP}{\mathfrak{sp}}

\DeclareMathOperator{\Tr}{Tr}

\DeclareMathOperator{\ad}{ad}

\DeclareMathOperator{\Sym}{Sym}
\newcommand{\trans}{\dagger}

\theoremstyle{plain} 
\newtheorem{thm}{Theorem}[]
\newtheorem{cor}{Corollary}[]

\newtheorem{prop}{Proposition}[]
\theoremstyle{definition}
\newtheorem{defn}{Definition}[]

\newtheorem{rem}{Remark}[] 
\newtheorem{ass}{Assumption}[]

\title{Lie--Poisson methods for isospectral flows}
\date{\today, Communicated by Arieh Iserles}

\author{Klas Modin}
\author{Milo Viviani}
\address{Department of Mathematical Sciences, Chalmers University of Technology and University of Gothenburg, SE-412 96 Gothenburg, Sweden}
\email{klas.modin@chalmers.se, viviani@chalmers.se}

\begin{document}

\maketitle

\begin{abstract}
The theory of isospectral flows comprises a large class of continuous dynamical systems, particularly integrable systems and Lie--Poisson systems.
Their discretization is a classical problem in numerical analysis.
Preserving the spectra in the discrete flow requires the conservation of high order polynomials, which is hard to come by.
Existing methods achieving this are complicated and usually fail to preserve the underlying Lie--Poisson structure.
Here we present a class of numerical methods of arbitrary order for Hamiltonian and non-Hamiltonian isospectral flows, which preserve both the spectra and the Lie--Poisson structure.
The methods are surprisingly simple, and avoid the use of constraints or exponential maps.
Furthermore, due to preservation of the Lie--Poisson structure, they exhibit near conservation of the Hamiltonian function.
As an illustration, we apply the methods to several classical isospectral flows.
\\[1ex]
\textbf{Keywords:} isospectral flow, Lie--Poisson integrator, symplectic Runge--Kutta methods, Toda flow, generalized rigid body, Chu's flow, Bloch--Iserles flow, Euler equations, point-vortices
\\[1ex]
\textbf{MSC 2010:} 37M15, 65P10, 37J15, 53D20, 70H06
\end{abstract}

\tableofcontents

\newpage
\section{Introduction}

Lie--Poisson systems and isospectral flows are two well-studied classes of dynamical systems. 
The former appear as Poisson reductions of Hamiltonian systems for which the configuration and symmetry space is a Lie group (see the monograph \cite{MaRa1999} and references therein). 
The classical example is the free rigid body as viewed by Poincaré~\cite{Po1901}.
The latter, isospectral flows, appear as Lax formulations of integrable systems (see the survey papers \cite{Wa1984,Ch1994,To2013} and references therein).
The classical example is the Toda lattice as viewed by Flaschka \cite{To1970,Fl1974}.

The study of numerical methods for the two classes of systems are by now classical subjects in numerical analysis.
The motivation for such schemes came through the strong connection between matrix factorizations in numerical linear algebra and isospectral flows (see the survey papers \cite{Ch2008b,Mo2017}).
This was initiated by the remarkable discovery that the iterative $QR$-algorithm for computing eigenvalues is a discretization of the (non-periodic) Toda flow \cite{Sy1982,DeNaTo1983}.

The general form of an isospectral flow is
\begin{equation}\label{iso_gen}
	\dot W = [B(W),W], \quad W\in S\subset\GL(n,\Cc).
\end{equation}
Here, $[\cdot,\cdot]$ denotes the matrix commutator, $S$ is a linear subspace of the Lie algebra $\GL(n,\Cc)$, and the function $B\colon S\to \NN(S)$ maps into the normalizer algebra $\NN(S)$ (see Section~\ref{sec:reduction} below for details).
The most studied setting is when $S=\Sym(n,\Rr)$ is the space of symmetric real matrices, for which the normalizer is the Lie algebra of skew-symmetric real matrices $\NN(S) = \SO(n)$.
Another setting is when $S=\GG$ is a Lie subgroup of $\GL(n,\Cc)$, for which the normalizer is the subalgebra itself $\NN(S) = \GG$.
 
Let us now discuss the connection between isospectral flows and Lie--Poisson systems.
The predominant example connecting the two is Manakov's $n$-dimensional rigid body \cite{man}.

Recall that a Lie--Poisson system evolves on the dual~$\GG^*$ of a Lie algebra~$\GG$.
Given a Hamiltonian function $H$ on $\GG^*$, the flow $W(t)\in\GG^*$ is given by
\begin{equation}\label{eq:gen_LP_abstract}
	\dot W = \ad^*_{d H(W)}(W),
\end{equation}
where the operator $\ad^*$ is defined by
\begin{equation}\label{eq:adstar_def}
	\langle \ad^*_{U}(W),V\rangle = \langle W,[U,V]\rangle\qquad \forall\, U,V\in \GG.
\end{equation}

Without loss of generality we may assume that $\GG$ is a subalgebra of $\GL(n,\Cc)$.
To identify $\GL(n,\Cc)^*$ with $\GL(n,\Cc)$ we use the Frobenius inner product 
\begin{equation*}
	\langle W,V\rangle = \Tr(W^\trans V),
\end{equation*}
where $W^\trans$ denotes the conjugate transpose.
In this way we also identify $\GG^*$ with the subspace $\GG\subset \GL(n,\Cc)$.
Next we extend the Hamiltonian to all of $\GL(n,\Cc)$ by taking it to be constant on the affine spaces given by translations of the orthogonal complement of $\GG$.
Then $d H$ corresponds to $\nabla H$.
From the definition \eqref{eq:adstar_def} and the identification of $\GG^*$ with $\GG$ we get
\begin{equation*}
	\ad^*_W(M) = \Pi\,[W^\trans,M],
\end{equation*}
where $\Pi$ is the orthogonal projection $\GL(n,\Cc)\to\GG$.
We thus arrive at an explicit formulation of the Lie--Poisson system~\eqref{eq:gen_LP_abstract}, namely
\begin{equation}\label{eq:gen_LP}
	\dot W = \Pi\,[\nabla H(W)^\trans,W].
\end{equation}

Now, the key observation is that if the representation of $\GG$ as a subalgebra of $\GL(n,\Cc)$ is closed under conjugate transpose, then equation~\eqref{eq:gen_LP} becomes the isospectral flow
\begin{equation}\label{eq:ham_isospectral}
	\dot W = [\nabla H(W)^\trans,W].
\end{equation}
Such a representation is possible if and only if $\GG$ is a \emph{reductive Lie algebra} (see Section~2--3 below for details).
Thus we arrive at the statement that Lie--Poisson systems for any reductive 
Lie algebra can be viewed as isospectral flows.
Recall that most classical Lie algebras are reductive, for example $\GL(n,\Cc), \GL(n,\Rr), \SL(n,\Cc)$, $\SL(n,\Rr)$, $\U(n)$, $\SU(n)$, $\SO(n)$, and $\SP(n)$. 

An interesting consequence of equation~\eqref{eq:ham_isospectral} is that whenever the function $B(W)$ in the isospectral flow \eqref{iso_gen} can be written as $B(W) = \nabla H(W)^\trans$, then it can be extended to a Lie--Poisson system on $\GL(n,\Cc)$ (or possibly a smaller reductive Lie algebra containing $S$).
Indeed, just extend the Hamiltonian function $H$ to be constant of the affine fibres orthogonal to $S$.
In this way we obtain an extended system foliated into invariant affine subspaces generated by $S$. 
The Toda flow is an example where this construction is possible (see Section~\ref{sec:todalattice} below).

The key feature of isospectral flows is, of course, that the eigenvalues of $W$ are preserved. 
Equivalently, given any analytic function $f$ extended to matrices, the function
\begin{equation*}
	F(W) = \Tr(f(W))
\end{equation*}
is a first integral regardless of the choice of $B(W)$ in \eqref{iso_gen}.
From the perspective of Lie--Poisson systems \eqref{eq:ham_isospectral}, this means that $F(W)$ is a \emph{Casimir function} associated with the Lie--Poisson structure~\eqref{eq:adstar_def}.
Although there are infinitely many Casimir functions, only a finite number of them can be functionally independent.

In this paper, we develop spectral preserving numerical methods for flows of the form~\eqref{iso_gen} which, in the case of Hamiltonian isospectral flows~\eqref{eq:ham_isospectral}, also preserves the Lie--Poisson structure.
There already exist at least four ways to achieve this:
\begin{itemize}
	\item If the Hamiltonian can be written as a sum of explicitly integrable Hamiltonians one can use \emph{splitting method} (see \cite{McQu2002} and references therein).

	\item The Lie--Poisson system on $\GG^*\simeq \GG$ can be extended to a constrained canonical Hamiltonian system on $T^*G\simeq TG\subset T\mathrm{GL}(n,\Cc)$.
	One can then use the symplectic RATTLE method (or higher order versions of it) for the constrained system (see \cite{Ja1996,McMoVeWi2013}).

	\item One can use symplectic Lie group methods on $T^*G$ as developed in~\cite{bogmar}. 
	These methods rely on an invertible mapping between the Lie algebra and (an identity neighbourhood of) the Lie group, such as the exponential map (works in general) or the Cayley map (works for quadratic Lie groups).

	\item One can, in some cases, use \emph{collective symplectic integrators}, which rely on Clebsch variables originating from a Hamiltonian action of $G$ on a symplectic vector space (see \cite{McMoVe2014c,McMoVe2014d} for details).
\end{itemize}
Compared to these methods our approach is:
(i) simpler since the algorithms are formulated directly on the algebra $\GG\subset \GL(n,\Cc)$;
(ii) free of constraints;
(iii) free of algebra-to-group maps, such as the exponential or Cayley map;
(iv) generic as they apply to any isospectral Hamiltonian flow.
Furthermore, through the framework of Poisson reduction (cf.~\cite{MaRa1999}) our methods are directly related to classical symplectic Runge--Kutta methods (or partitioned symplectic Runge--Kutta methods).
Therefore they merit the designation \emph{Isospectral Symplectic Runge--Kutta} (IsoSyRK) methods.

The paper is organized as follows. 
In Section~\ref{sec:main_results} we give the definitions of the new methods and we state our main results.
In Section~\ref{sec:reduction} we develop a discrete reduction theory for isospectral Lie--Poisson flows.
These results are instrumental in Section~\ref{sec:isosyrk}, where we specialize our construction to symplectic Runge--Kutta methods.
All numerical examples are given in Section~\ref{sec:examples}.

\medskip

\noindent\textbf{Acknowledgements.} 
The authors were supported by EU Horizon 2020 grant No 691070, by the Swedish Foundation for International Cooperation in Research and Higher Eduction (STINT) grant No PT2014-5823, by the Swedish Foundation for Strategic Research grant ICA12-0052, and by the
Swedish Research Council (VR) grant No 2017-05040.
We would like to thank the reviewers for helpful comments.



\section{Main results}\label{sec:main_results}


A Runge--Kutta method is defined by its \emph{Butcher tableau} (cf.~\cite{hlw})
\begin{equation}\label{eq:butchertab}
	\renewcommand\arraystretch{1.2}
	\begin{array}
	{l|l}
	\mathbf{c}  & \mathbf{A} \\
	\hline
	& \mathbf{b}^T  
	\end{array}
\end{equation}
where $\mathbf{A}\in \Rr^{s\times s}$ and $\mathbf{b},\mathbf{c} \in \Rr^s$.
Furthermore, if 
\begin{equation}\label{eq:sympl_condition}
	b_ia_{ij}+b_ja_{ji}= b_ib_j,   
\end{equation}
for $i,j=1,\ldots,s$, then the corresponding Runge--Kutta method is symplectic when applied to canonical Hamiltonian systems on $\Rr^{2n}$ \cite{Sa1988}.
However, directly applying a symplectic Runge--Kutta method to the Hamiltonian isospectral flow~\eqref{eq:ham_isospectral} does \emph{not} yield a Poisson integrator.
Nor does it, in general, preserve the isospectral property, as is well known.

\begin{defn}[IsoSyRK]\label{def:isosyrk}
	Given a Butcher tableau~\eqref{eq:butchertab}
	fulfilling the symplectic condition~\eqref{eq:sympl_condition},
	the corresponding \emph{Isospectral Symplectic Runge--Kutta method} for the flow \eqref{iso_gen} is the map
	\begin{equation*}
		\Phi_{h}\colon\GL(n,\Cc)\ni W_k \longmapsto W_{k+1}\in \GL(n,\Cc)
	\end{equation*}
	defined by
	\[
		\left. \begin{aligned} 
		X_i &= - \Big(W_k+\sum_{j=1}^s a_{ij} X_j\Big)h B(\widetilde{W}_i)  
		\\
		Y_i &=  hB(\widetilde{W}_i)\Big(W_k+\sum_{j=1}^s a_{ij} Y_j\Big)  
		\\
		K_{ij} &= hB(\widetilde{W}_i)\Big(\sum_{j'=1}^s (a_{ij'}X_{j'}+a_{jj'}K_{ij'})\Big)
		\\
		\widetilde{W}_i&=W_k+\sum_{j=1}^s a_{ij} (X_j+Y_j+K_{ij})  \hspace{1cm}  & 
		\\
		W_{k+1} &= W_k + \sum_{i=1}^sb_i[h B(\widetilde{W}_i),\widetilde{W}_i] ,
		\end{aligned}\right.
	\]
	for $i,j=1,\ldots,s$, where $h>0$ denotes the step size.
\end{defn} 

\begin{thm}\label{thm:main_isosyrk}
	The method in Definition~\ref{def:isosyrk} fulfills the following properties:
	\begin{enumerate}
		\item It has the same order as the underlying Runge--Kutta method.

		\item It is isospectral; for any analytic function $f$ extended to matrices
		\begin{equation*}
			\Tr(f(W_{k+1})) = \Tr(f(W_k)).
		\end{equation*}

		\item It is equivariant with respect to Lie algebra morphisms; \linebreak if $\mathcal A\colon \GL(n,\Cc)\to \GL(n,\Cc)$ is a linear invertible mapping fulfilling for all $X,Y\in\GL(n,\Cc)$
		\begin{equation*}
			\mathcal A[X,Y] = [\mathcal A X,\mathcal A Y],
		\end{equation*}
		then the following diagram commutes
		\[
		\begin{tikzcd}
			W_{k} \arrow[r,"\mathcal A"]\arrow[d,"\Phi_{h B}"] & W_k'\arrow[d,"\Phi_{h \mathcal A\circ B\circ\mathcal A^{-1}}"] \\
			W_{k+1} \arrow[r,"\mathcal A"] & W'_{k+1}
		\end{tikzcd}
		\]

		\item It is a Lie--Poisson integrator if the isospectral flow is Hamiltonian, i.e., of the form \eqref{eq:ham_isospectral}. Furthermore, if the isospectral flow is Hamiltonian, it is equivariant with respect to linear Lie--Poisson isomorphisms $\mathcal B\colon \GL(n,\Cc)^*\to \GL(n,\Cc)^*$. 

		\item It restricts to a Lie--Poisson integrator for any Lie subalgebra $\GG\subset \GL(n,\Cc)$ defined by
		\begin{equation}\label{eq:quadratic_alg1}
			W\in \GG \iff W^\dagger J + J W = 0,
		\end{equation}
		where $J^2 = c I $ for some $c\in \mathbb{R}\backslash \{ 0\}$.

		\item It restricts to a Lie--Poisson integrator for any Lie subalgebra given by arbitrary intersections of $\GL(n,\Rr)$, $\SL(n,\Cc)$, and Lie algebras of the form~\eqref{eq:quadratic_alg1}.\label{item:intersections}

		\item It extends to a Lie--Poisson integrator for direct products of Lie algebras of the form in item \eqref{item:intersections}.

		\item It restricts to an isospectral integrator on the orthogonal complement $\GG^\bot \subset \GL(n,\Cc)$ of any Lie algebra $\GG$ of the form in item \eqref{item:intersections}, provided that $B$ restricts to a mapping $B\colon\GG^\bot\to \GG$.
	\end{enumerate}
\end{thm}

\begin{proof}
	The theorem is a combination of results proved in Theorem~\ref{thm1}, Corollary~\ref{cor1}, Theorem~\ref{thm_RK1}, and Theorem~\ref{thm:affine_equiv} below.
\end{proof}

\begin{rem}
	Items (5)--(6) of Theorem~\ref{thm:main_isosyrk} implies that the IsoSyRK methods constitute Lie--Poisson integrators for the classical Lie algebras $\SL(n,\Cc)$, $\SL(n,\Rr)$, $\SO(n)$, $\U(n)$, $\SU(n)$, $\SP(n,\Cc)$, and $\SP(n,\Rr)$.
	Item (8) implies that they also preserve the classical isospectral setting as flows on symmetric or Hermitian matrices, since, for example, $\SO(n)^\bot = \Sym(n,\Rr)$.
\end{rem}

We also have an analogous, albeit slightly weaker, result for partitioned symplectic Runge--Kutta methods, such as defined by 
two Butcher tableaux (cf.~\cite{hlw})
\begin{equation}\label{eq:butchertab_part}
	\renewcommand\arraystretch{1.2}
	\begin{array}
	{c|c}
	\mathbf{c}  & \mathbf{A} \\
	\hline
	& \mathbf{b}^T  
	\end{array}
	\hspace{2cm}
	\renewcommand\arraystretch{1.2}
	\begin{array}
	{c|c}
	\mathbf{\widehat{c}}  & \mathbf{\widehat{A}} \\
	\hline
	& \mathbf{\widehat{b}}^\top 
	\end{array}
	\; .
\end{equation}
If, for $i,j=1,\ldots,s$, the coefficients in the tableaux fulfill 
\begin{equation}\label{eq:sympl_condition_part}
\begin{array}{ll}
b_i\widehat{a}_{ij}+\widehat{b}_ja_{ji}=b_i\widehat{b}_j, \\
\widehat{b}_i=b_i,
\end{array}
\end{equation}
then the corresponding partitioned Runge--Kutta method is symplectic when applied to canonical Hamiltonian systems on $\Rr^{2n}$.

\begin{defn}[IsoSyPRK]\label{def:isosyrk_part}
	Given two Butcher tableaux~\eqref{eq:butchertab_part}
	fulfilling the symplectic conditions~\eqref{eq:sympl_condition_part},
	the corresponding \emph{Isospectral Symplectic Partitioned Runge--Kutta method} for the flow \eqref{iso_gen} is the map
	\begin{equation*}
		\Phi_{h}\colon\GL(n,\Cc)\ni W_k \longmapsto W_{k+1}\in \GL(n,\Cc)
	\end{equation*}
	defined by
\begin{align*}
X_i &= - h\Big(W_k+\sum_{j=1}^s a_{ij} X_j\Big)B(\widetilde{W}_i)   
\\
Y_i &=  hB(\widetilde{W}_i)\Big(W_k+\sum_{j=1}^s \widehat{a}_{ij} Y_j\Big)  
\\
K_{ij} &= hB(\widetilde{W}_i)\Big(\sum_{j'=1}^s (a_{ij'}X_{j'}+\widehat{a}_{jj'}K_{ij'})\Big)
\\
\widetilde{W}_i&=W_k+\sum_{j=1}^s a_{ij} X_j+\widehat{a}_{ij} (Y_j+K_{ij}) 
\\
W_{k+1} &= W_k + h\sum_{i=1}^sb_i[B(\widetilde{W}_i),\widetilde{W}_i] .
\end{align*} 
	for $i,j=1,\ldots,s$, where $h>0$ denotes the step size.
\end{defn} 

\begin{thm}\label{thm:main_isosyrk_partitioned}
	The method in Definition~\ref{def:isosyrk_part} fulfills the following properties:
	\begin{enumerate}
		\item It has the same order as the underlying partioned Runge--Kutta method.

		\item It is isospectral; for any analytic function $f$ extended to matrices
		\begin{equation*}
			\Tr(f(W_{k+1})) = \Tr(f(W_k)).
		\end{equation*}

		\item It is equivariant with respect to Lie algebra morphisms; \linebreak if $\mathcal A\colon \GL(n,\Cc)\to \GL(n,\Cc)$ is a linear invertible mapping fulfilling for all $X,Y\in\GL(n,\Cc)$
		\begin{equation*}
			\mathcal A[X,Y] = [\mathcal A X,\mathcal A Y],
		\end{equation*}
		then the following diagram commutes
		\[
		\begin{tikzcd}
			W_{k} \arrow[r,"\mathcal A"]\arrow[d,"\Phi_{h B}"] & W_k'\arrow[d,"\Phi_{h \mathcal A\circ B\circ\mathcal A^{-1}}"] \\
			W_{k+1} \arrow[r,"\mathcal A"] & W'_{k+1}
		\end{tikzcd}
		\]

		\item It is a Lie--Poisson integrator if the isospectral flow is Hamiltonian, i.e., of the form \eqref{eq:ham_isospectral}. Furthermore, if the isospectral flow is Hamiltonian, it is equivariant with respect to linear Lie--Poisson isomorphisms $\mathcal B\colon \GL(n,\Cc)^*\to \GL(n,\Cc)^*$. 

		\item If it restricts to a Lie--Poisson integrator for a Lie subalgebra $\GG\subset \GL(n,\Cc)$ defined by
		\begin{equation*}\label{eq:quadratic_alg}
			W\in \GG \iff W^\dagger J + J W = 0,
		\end{equation*}
		where $J^2 = c I $ for some $c\in \mathbb{R}\backslash \{ 0\}$ and $b_i\neq 0$ for $i=1,\ldots,s$, then the two Butcher tableaux coincide (it is a standard Runge--Kutta method).

	\end{enumerate}
\end{thm}

\begin{proof}
	The theorem is a combination of results proved in Theorem~\ref{thm1}, Corollary~\ref{cor1}, Theorem~\ref{thm_RK2}, and Theorem~\ref{thm:affine_equiv} below.
\end{proof}

\section{Reduction theory for isospectral Lie--Poisson integrators}\label{sec:reduction}

Let us consider Lie--Poisson systems of the form \eqref{eq:ham_isospectral}.
Conditions under which the flow remains in a linear subspace $S$ of $\GL(n,\Cc)$ are intrinsically connected to the $\GL(n,\Cc)$-normalizer of S. 
We recall here its definition:

\begin{defn}\label{def_normalizer}
Let $G$ be a Lie group and $\GG$ its Lie algebra. 
Furthermore, let $S\subseteq \GG$ be a linear subspace. 
Then the two sets
\begin{align*}
&N(S)=\lbrace g\in G \mid g^{-1} S g\subseteq S \rbrace\\
&\NN(S)=\lbrace \xi\in \GG \mid [\xi,S]\subseteq S \rbrace
\end{align*}
are respectively called the $G$-normalizer and the $\GG$-normalizer of $S$. 
Notice that $N(S)$ is a subgroup of $G$ and $\NN(S)$ is a Lie subalgebra of $\GG$.
\end{defn}

We now give some examples of Definition~\ref{def_normalizer}. We first give the following definition.

\begin{defn}
Let $\GG\subseteq\SL(n,\Cc)$ be a Lie algebra and $J\in GL(n,\Cc)$.  Then $\GG$ is said to be a $J$-quadratic Lie algebra if $A^\trans J+JA=0$, for any $A\in\GG$.
\end{defn}

\noindent{Examples of $\GL(n,\Cc)$-normalizers:}
\begin{enumerate}
\item $S=\SL(n,\Cc)$ with $\NN(S)=\GL(n,\Cc)$.

\item $S=\GG\subset\SL(n,\Cc)$ is a $J$-quadratic Lie subalgebra with $\NN(S)=\GG\oplus\Cc\, \mathrm{Id}$.
A typical case is $S=\SU(n)$ for which $\NN(S)=\mathfrak{su}(n)\oplus\Cc\, \mathrm{Id}$, corresponding to~$J=\mathrm{Id}$.

\item $S=\GG^\bot$, where $\GG\subset\SL(n,\Cc)$ is a $J$-quadratic Lie subalgebra, with $\NN(S)=\GG\oplus\Cc\, \mathrm{Id}$.\footnote{Orthogonal complements are taken with respect to the Frobenius inner product. We always have that $\NN(S^\bot)=\NN(S)^\trans$.}\
Restricting to $\GL(n,\Rr)$, a typical case is $S=\Sym(n,\Rr)$ and $\NN(S)=\mathfrak{o}(n)\oplus\Rr\, \mathrm{Id}$, corresponding to~$J=\mathrm{Id}$.
\end{enumerate}

\begin{rem}
If $S= \GG$ is a Lie subalgebra of $\GL(n,\Cc)$ one may ask under which conditions the isospectral Hamiltonian system~\eqref{eq:ham_isospectral} coincide with the Lie--Poisson system on $\GG$.
Recall that $\nabla H$ is the gradient of $H$ with respect to the Frobenius inner product.
It is not a restriction to assume that $\nabla H(W)\in \GG$ for all $W\in \GL(n,\Cc)$ (since we can extend $H$ to be constant on the affine complements of $\GG$).
Due to the conjugate transpose on $\nabla H(W)$ it is not, however, enough that $\nabla H(W) \in \GG$; instead we need $\nabla H(W)^\dagger \in \GG$.
A sufficient condition for this to be true is that $\GG$ is closed under conjugate transpose: $\GG^\dagger \subset \GG$.
Such $\GG$ are, up to representation, the \emph{semisimple Lie algebras} (\cite{kna}, Prop. 6.28). 
This means that, after the identification of the dual of the Lie algebra $\GG$ with itself (using the Frobenius inner product), a Lie--Poisson system on a semisimple Lie algebra $\GG$ coincides with a Lie--Poisson system on $\GL(n,\Cc)$ restricted to $\GG$. 
In fact, slightly more is true: due to the bracket in the right hand side of \eqref{eq:ham_isospectral} it is enough that
\[ [\GG^\dagger,\GG]\subset\GG.
\]
This holds when $\GG$ is a \emph{reductive} Lie algebra, i.e., the direct sum of a semisimple Lie 	algebra and an abelian Lie algebra.
\end{rem}

A Lie--Poisson systems on a Lie algebra $\GG$ can be viewed as the \emph{Lie--Poisson reduction} of a canonical Hamiltonian system on $T^*G$ with a $G$-symmetric Hamiltonian.
Going backwards, one may `unreduce' any Lie--Poisson system to a canonical Hamiltonian system on the cotangent bundle of the corresponding Lie group.
Our objective is to show that the Isospectral Symplectic Runge--Kutta methods (cf.~Section~\ref{sec:main_results}) originates as a ``discrete Lie--Poisson reduction'' of symplectic Runge--Kutta methods. 

We thus proceed by extending the equations \eqref{eq:ham_isospectral} to a canonical system on\linebreak $T^*GL(n,\Cc)$. 
To do so, one needs the \emph{momentum map} (cf.~\cite{MaRa1999}) associated with the right action of $GL(n,\Cc)$ on $T^*GL(n,\Cc)$:
\begin{equation*}\label{eq:cotangentaction_right}
(Q,P)\cdot G = (QG, P(G^{-1})^\dagger),
\end{equation*}
for $G,Q\in GL(n,\Cc)$ and $P\in T^*_QGL(n,\Cc)$.
The momentum map for this (Hamiltonian) action is given by
\begin{equation*}
	\mu\colon T^*GL(n,\Cc) \to \GL(n,\Cc)^*\simeq \GL(n,\Cc), \quad \mu(Q,P)=Q^\trans P.
\end{equation*}

This momentum map provides a left-invariant Hamiltonian function $\widetilde{H}(Q,P)=H(Q^\dagger P)$ on $T^*GL(n,\Cc)$, i.e., an Hamiltonian function invariant with respect to the left action
\begin{equation}\label{eq:cotangentaction}
G\cdot(Q,P) = (GQ,(G^{-1})^\dagger P).
\end{equation}

The fact that the momentum map is a Poisson map between $T^*GL(n,\Cc)$ and $\GL(n,\Cc)^*$ means that a symplectic map in $\Phi\colon T^*GL(n,\Cc)\to T^*GL(n,\Cc)$ which is equivariant with respect to the action \eqref{eq:cotangentaction} \emph{descends} to a corresponding map $\phi\colon \GL(n,\Cc)^* \to \GL(n,\Cc)^*$.
In terms of numerical integrators, this means that a $GL(n,\Cc)$-equivariant symplectic integrator on $T^*GL(n,\Cc)$ induces a Poisson integrator on $\GL(n,\Cc)^*$.
As we shall see, this is precisely how the isospectral symplectic Runge--Kutta methods come about. 


Using the momentum map~\eqref{eq:cotangentaction}, the canonical Hamiltonian system on $T^*GL(n,\Cc)$ is given by
\begin{equation}\label{ham1}
\begin{array}{ll} 
	\dot{Q}= Q\nabla H(Q^\trans P) \\
	\dot{P}= -P\nabla H(Q^\trans P)^\trans,
			\end{array}
\end{equation}
where $H$ is the same Hamiltonian as in \eqref{eq:ham_isospectral}.

We now translate the condition of staying on $S$ from \eqref{eq:ham_isospectral} to \eqref{ham1}.
\begin{prop}\label{prop_alg_cons}
Consider a solution $(Q(t),P(t))$ of Hamilton's equations~\eqref{ham1} for a given initial point $(Q(0),P(0))$ and let $S\subseteq\GL^*(n,\Cc)$ be a linear subspace as before. 
Then there exists a time $T>0$ such that the following three statements are equivalent:
\begin{enumerate}
\item $Q(t)^\trans P(t)\in S$, for any $0\leq t\leq  T$;
\item $Q(0)^\trans P(0)\in S$ and $\nabla H(Q^\trans P)^\trans (t)\in\NN(S)$, for any $0\leq t\leq  T$;
\item $Q(0)^\trans P(0)\in S$ and there exists a fixed $G\in GL(n,\Cc)$ such that $GQ(t)^\dagger\in N(S)$, for any $0\leq t\leq  T$.
\end{enumerate}
\end{prop}
\proof 
Let $U$ be a neighbourhood of $Q(0)^\trans P(0)$ such that the map $\exp^{-1}:U\subset GL(n,\Cc)\rightarrow\GL(n,\Cc) $ is well defined. Then, let $T$ be a positive real number such that $\exp(\int_{0}^{t} \nabla H(Q^\trans P)^\trans (s) ds)\in U$, for $0\leq t\leq T$.

$1)\Rightarrow 2)$ We have that $S\ni\frac{dQ^\trans P}{dt}=[\nabla H(Q^\trans P)^\trans ,Q^\trans P]$, since $S$ is a linear space and $(Q^\trans P)(t)\in S$, for any $0\leq t\leq T$. But this means that $\nabla H(Q^\trans P)^\trans $ has to be in $\NN(S)$, for any $0\leq t\leq  T$. 

$2)\Rightarrow 1)$ For $0\leq t\leq T$, we have that:
\begin{center}
$(Q^\trans P)(t)=\exp(\int_{0}^{t} \nabla H(Q^\trans P)^\trans (s) ds)Q(0)^\trans P(0)\exp(-\int_0^t \nabla H(Q^\trans P)^\trans (s) ds),$
\end{center}
which proves the statement, since $N(S)\supseteq\exp(\NN(S))$.

$2)\Rightarrow 3)$ Let $G\in GL(n,\Cc)$ such that $GQ(t)^\dagger\in N(S)$. Then we have:
\begin{center}
$GQ(t)^\trans=\exp(\int_{0}^{t} \nabla H(Q(s)^\trans P(s))^\trans  ds)GQ(0)^\trans,$
\end{center}
which proves the statement, since $N(S)\supseteq\exp(\NN(S))$.

$3)\Rightarrow 2)$ By the formula above we have
\begin{center}
$GQ^\trans (t)(GQ(0)^\trans)^{-1}=\exp(\int_{0}^{t} \nabla H(Q(s)^\trans P(s))^\trans ds)$.
\end{center}
Since the left-hand side is in $N(S)$ for any $0\leq t\leq  T$,
we have $\nabla H(Q(t)^\trans P(t))^\trans \in\NN(S)$ for any $0\leq t\leq  T$, by the definition of $T$.
\endproof

Although the statements in Proposition~\ref{prop_alg_cons} are equivalent for the exact flow, they are different after discretization.
Indeed, in order to understand the conditions for our isospectral symplectic Runge--Kutta methods to preserve the flow on $S$ we need the definition of weak and strong first integrals.

\begin{defn}
Let $M$ be a smooth manifold and $N\subset M$ a smooth submanifold. 
Consider the following dynamical system on $N$:
\begin{equation}\label{dyn_sys}
\begin{array}{ll} 
	\dot{z}= X(z) \\
	z(0) = z_0,
\end{array}
\end{equation}
with $X$ a smooth vector field on $N$ and $z_0\in N$. 
Assume further that $X$ can be extended on a $\varepsilon-$neighbourhood $N_\varepsilon$ of $N$ in $M$.

Then a differentiable function $I\colon N_\varepsilon\rightarrow\Cc$ is said to be a \emph{weak}, respectively, \emph{strong} first integral of (\ref{dyn_sys}) if
\begin{equation*}
\begin{array}{ll} 
	\langle d I(z), X(z)\rangle = 0 \mbox{ for all }z\in N \\
	\langle d I(z), X(z)\rangle = 0 \mbox{ for all }z\in N_\varepsilon.
\end{array}
\end{equation*}

\end{defn}

In numerical analysis it is often the case that integration schemes on a submanifold $N$ actually depends on how $N$ is embedded in a larger (vector) space $M$.
That is, the integration scheme is \emph{not} intrinsic to $N$ (for example evaluations of the vector field outside of $N$ may occur).
In this situation, the difference between strong and weak first integrals is essential.
Indeed, for non-intrinsic methods one can at best expected to conserve strong first integrals.
%
Motivated by this we make the following:
\begin{ass}\label{ass1}
Let $S_\varepsilon$ be a $\varepsilon-$neighbourhood of $S$ in $\GL(n,\Cc)$. 
We assume that $\nabla H^\trans$ can be extended to $S_\varepsilon$ such that $\nabla H(W)^\trans \in\NN(S)$ for all $W\in S_\varepsilon$. 
\end{ass}
Since $S$ is a linear space the natural way to extend $\nabla H^\trans$ is to take it to be constant on the affine complements of $S$.
With this extension the gradient of the Hamiltonian requires only an orthogonal projection of $W$ to $S$.

Under Assumption~\ref{ass1} our Proposition~\ref{prop_alg_cons} says that $(Q^\trans P)\in S$ is determined by \textit{weak first integrals} 
of the Hamiltonian system \eqref{ham1} provided that the gradient of the Hamiltonian is in $\NN(S)$. 
In fact, having $(Q^\trans P)\in S$ is equivalent to $[\nabla H(Q^\trans P)^\trans ,Q^\trans P]\in S$ which in general is not true for $Q^\trans P$ in an $\varepsilon-$neighbourhood of $S$. 
Instead an equivalent formulation corresponding to \textit{strong first integrals}
is given by the third statement, which says that there exists a fixed matrix $G$ such that $GQ^\trans \in N(S)$. 
Therefore only the numerical methods that have $GQ^\trans \in N(S)$ as a discrete invariant correspond to integrators that preserve~$S$.
In particular, if $N(S)$ is a quadratic Lie group one can expect symplectic Runge--Kutta methods to yield a discrete flow that preserves $S$ since they preserve general quadratic first integrals.
On the other hand, the same cannot be expected from symplectic partitioned Runge--Kutta methods, since they only preserve special (bilinear) quadratic first integrals.


We summarize our findings in the following theorem.
\begin{thm}\label{thm1}
Consider a Lie--Poisson system of the form \eqref{eq:ham_isospectral} evolving on a linear subspace $S\subset\GL(n,\Cc)$.
Let $\Phi_h\colon T^*GL(n,\Cc)\to T^*GL(n,\Cc)$ be a symplectic numerical method for the corresponding canonical Hamiltonian system \eqref{ham1} obtained by extension from $S$ in accordance with Assumption~\ref{ass1}.
\begin{enumerate}
	\item If $\Phi_h$ is equivariant with respect to the action \eqref{eq:cotangentaction}, i.e.,
\[
	G\cdot\Phi_h(Q,P) =\Phi_h( G\cdot(Q,P)). 
\]
then it descends to a Lie--Poisson integrator $\phi_h$ on $\GL(n,\Cc)$.
	\item If, in addition, $\Phi_h$ preserves the foliation 
	\[
		\mathcal F_G = \{Q\mid GQ^\trans \in N(S) \}, \qquad G\in GL(n,\Cc)
	\]
	then $\phi_h$ restricts to an integrator on $S$.
\end{enumerate}
\end{thm} 

Based on the results in Theorem~\ref{thm1}, we can now generalize the results to a general $B(\cdot)$, i.e., to isospectral flows that are not necessarily Hamiltonian. 
This extension requires that the underlying method can be expanded in a \emph{B-series} or \emph{P-series} (cf.~\cite{hlw} for definitions and notation).\footnote{Please notice the following clash of notation: \emph{B-series} and \emph{P-series} have nothing to do with the function $B$ and the variable $P$ as defined in this paper.}\
Consider first the generalization of Assumption~\ref{ass1}:
\begin{ass}\label{ass2}
Let $S_\varepsilon$ be a $\varepsilon-$neighbourhood of $S$ in $\GL(n,\Cc)$. 
We assume that $B(\cdot)$ can be extended to $S_\varepsilon$ such that $B(W) \in\NN(S)$ for all $W\in S_\varepsilon$. 
\end{ass}
Then, based on Theorem~\ref{thm1}, we have the following result.
\begin{cor}\label{cor1}
Consider an isospectral flow of the form \eqref{iso_gen} evolving on a linear subspace $S\subset\GL(n,\Cc)$.
Let $\Phi_h\colon T^*GL(n,\Cc)\to T^*GL(n,\Cc)$ be a symplectic B-series (or P-series) method for the corresponding system:
\begin{equation}\label{eq:can_B}
\begin{array}{ll} 
	\dot{Q}= QB(Q^\trans P)^\trans \\
	\dot{P}= -PB(Q^\trans P),
			\end{array}
\end{equation} obtained by extension from $S$ in accordance with Assumption~\ref{ass2}.
\begin{enumerate}
	\item If $\Phi_h$ is equivariant with respect to the action \eqref{eq:cotangentaction}, i.e.,
\[
	G\cdot\Phi_h(Q,P) =\Phi_h( G\cdot(Q,P)). 
\]
then it descends to an isospectral integrator $\phi_h$ on $\GL(n,\Cc)$.
	\item If, in addition, $\Phi_h$ preserves the foliation 
	\[
		\mathcal F_G = \{Q\mid GQ^\trans \in N(S) \}, \qquad G\in GL(n,\Cc)
	\]
	then $\phi_h$ restricts to an integrator on $S$.
\end{enumerate}
\end{cor}
\proof 
From Theorem~\ref{thm1} we know that for $B(W)=\nabla H(W)^\trans $ when we solve \eqref{eq:can_B} with a symplectic integrator the discrete flow is isospectral for $W:=Q^\trans P$. 
Therefore, the (truncated) modified equation is of the form
\begin{equation}\label{eq:can_B_ham}
\left. \begin{array}{ll} 
	\dot{Q} = Q \nabla\widetilde{H}(Q^\trans P)\\
	\dot{P} = - P \nabla\widetilde{H}(Q^\trans P)^\trans,
			\end{array}\right.
\end{equation} 
for some modified Hamiltonian $\widetilde{H}$. 
On the other hand, since $\Phi_h$ is a symplectic B-series method, the right hand side in \eqref{eq:can_B_ham} is a B-series whose coefficients satisfy the relation $b(u\circ v)+b(v\circ u)=0$ for each pair of trees $u,v$ \cite[ Theorem IX.9.3]{hlw}. 
In particular $b(u\circ u)=0$. 
From \cite[Lem~IX.9.6 and Thm~IX.9.8]{hlw} it follows that the only elementary Hamiltonians that vanish for all the Hamiltonian functions $H$ are those of the type $H(u\circ u)$.
Furthermore, it is clear from the form of \eqref{eq:can_B_ham} and the definition of B-series in terms of elementary differentials that the right hand side in \eqref{eq:can_B_ham} is of the form
\begin{equation}\label{eq:can_B_hamRHS}
\left( \begin{array}{ll} 
	Q \sum_{n=1}^\infty \sum_{k=1}^\infty A^k_n(Q^\trans P,H'(Q^\trans P),H''(Q^\trans P),\ldots,H^{(k)}(Q^\trans P))\\
	-P \sum_{n=1}^\infty \sum_{k=1}^\infty B^k_n(Q^\trans P,H'(Q^\trans P),H''(Q^\trans P),\ldots,H^{(k)}(Q^\trans P))^\trans,
			\end{array}\right)
\end{equation} 
for $A_n^k,B_n^k$ homogeneous polynomials of degree $n$ for each $k$.
 
We claim that to get the modified equation for a general $B$ we just replace in \eqref{eq:can_B_hamRHS} $(H^{(k)})^\trans $ with $B^{(k-1)}$ (which is possible since $k\geq 1$). 
Indeed, this follows since the coefficients of a symplectic B-series are uniquely determined by Hamiltonian vector fields \cite[Thm~IX.9.10]{hlw}.
Therefore, we conclude that a symplectic B-series integrator applied to \eqref{eq:can_B}, for a general $B$, is isospectral for $W=Q^\trans P$ with a modified equation of the form
\begin{equation}\label{eq:can_B_mod}
\left. \begin{array}{ll} 
	\dot{Q} = Q \widetilde{B}(Q^\trans P)^\trans  \\
	\dot{P} = - P \widetilde{B}(Q^\trans P),
			\end{array}\right.
\end{equation} 
for some $\widetilde{B}(\cdot)$ obtained by replacing in \eqref{eq:can_B_hamRHS} the $(H^{(k)})^\trans $ with $B^{(k-1)}$.

In the case when $\Phi_h$ is a symplectic P-series, the proof is repeated similarly, using instead \cite[Thm IX.10.3, Lem IX.10.6, Thm IX.10.8]{hlw}.
 \endproof
 
\section{Isospectral symplectic Runge--Kutta methods}\label{sec:isosyrk}

In this section we specialize Theorem~\ref{thm1} to the symplectic Runge--Kutta and partitioned Runge--Kutta methods. 
As a result, we obtain the novel numerical schemes for isospectral (Lie--Poisson) systems presented in Section~\ref{sec:main_results} above.

\subsection{Symplectic Runge--Kutta methods}

Given a Butcher tableau
\[
\renewcommand\arraystretch{1.2}
\begin{array}
{c|ccc}
c_1  & a_{11} &\dots & a_{1s} \\
\vdots & \vdots &\ddots & \vdots \\
c_s &a_{s1} &\dots & a_{ss}\\
\hline
& b_1 &\dots &b_s 
\end{array}
\]
the associated Runge--Kutta method for \eqref{iso_gen} is
\begin{equation}\label{eq:symplecticRK}
\begin{aligned} 
K^Q_i &= (Q_n+h\sum_{j=1}^s a_{ij} K^Q_j)B((Q_n+h\sum_{j=1}^s a_{ij} K^Q_j)^\trans (P_n+h\sum_{j=1}^s a_{ij} K^P_j))^\trans\\
\\
K^P_i &= - (P_n+h\sum_{j=1}^s a_{ij} K^P_j)B((Q_n+h\sum_{j=1}^s a_{ij} K^P_j)^\trans (P_n+h\sum_{j=1}^s a_{ij} K^P_j)) \\
\\ 
Q_{n+1} &= Q_n + h\sum_{i=1}^s b_i K^Q_i\\
\\
P_{n+1} &= P_n + h\sum_{i=1}^s b_i K^P_i,
\end{aligned}
\end{equation}
for $i,j=1,\ldots,s$. 
Recall that the method is symplectic, i.e., the discrete flow is a symplectic map, if $b_ia_{ij}+b_ja_{ji}=b_ib_j$ for any $i,j=1,\ldots,s$.

\begin{thm}\label{thm_RK1}
Given a Butcher tableau
\[
\renewcommand\arraystretch{1.2}
\begin{array}
{l|l}
\mathbf{c}  & \mathbf{A} \\
\hline
& \mathbf{b}^\top
\end{array}
\]
of a symplectic s-stages Runge--Kutta method, let $\Phi_h\colon T^*GL(n,\Cc) \to T^*GL(n,\Cc)$ denote the corresponding integrator map for the system~\eqref{ham1}.
Then:
\begin{enumerate}
\item The symplectic integrator $\Phi_h$ descends to a Lie--Poisson integrator $\phi_h$ on $\GL(n,\Cc)^*\simeq \GL(n,\Cc)$ for the isospectral Hamiltonian system~\eqref{eq:ham_isospectral}.
Furthermore, the map $\phi_h$ is completely constructive as an implicit integration scheme (see below for specific formulas).

\item If $S$ is an invariant subspace of \eqref{eq:ham_isospectral} (as described above), then $\phi_h$ preserves $S$ in the cases $S=\mathfrak{sl}(N,\mathbb{C})$, $S=\GG$, and $S=\GG^\bot$, for $\GG$ a $J$-quadratic Lie subalgebra.

\end{enumerate}


\end{thm}


The schemes obtained in Theorem~\ref{thm_RK1} are the following:

\subsubsection*{1. $S=\SL(n,\Cc)$ or $S=\GL(n,\Cc)$}
\[
\left. \begin{array}{llll} 
X_i = - h(W_n+\sum_{j=1}^s a_{ij} X_j)B(\widetilde{W}_i)  \\ 
\\
Y_i =  hB(\widetilde{W}_i) (W_n+\sum_{j=1}^s a_{ij} Y_j)  \\
\\
K_{ij} = hB(\widetilde{W}_i) (\sum_{j'=1}^s (a_{ij'}X_{j'}+a_{jj'}K_{ij'})) \\
\\
\widetilde{W}_i=W_n+\sum_{j=1}^s a_{ij} (X_j+Y_j+K_{ij}) \\
\\
W_{n+1} = W_n + h\sum_{i=1}^sb_i[B(\widetilde{W}_i),\widetilde{W}_i],
\end{array}\right. 
\]
for $i,j=1,\ldots,s$, where the unknowns are $X_i,Y_i,K_{ij}$ for $i,j=1,\ldots,s$ and the last two lines are explicit. 

\subsubsection*{2. $S=\GG\subseteq\GL(n,\Cc)$ $J$-quadratic}
\[
\left. \begin{array}{llll} 
X_i = - h(W_n+\sum_{j=1}^s a_{ij} X_j)B(\widetilde{W}_i)  \\
\\
K_{ij} = hB(\widetilde{W}_i) (\sum_{j'=1}^s (a_{ij'}X_{j'}+a_{jj'}K_{ij'}))\\
\\
\widetilde{W}_i=W_n+\sum_{j=1}^s a_{ij} (X_j-J^{-1}X_j^\trans J+K_{ij})  \\
\\
W_{n+1} = W_n + h\sum_{i=1}^sb_i[B(\widetilde{W}_i)^\trans ,\widetilde{W}_i], 
\end{array}\right. 
\]
for $i,j=1,\ldots,s$, where the unknowns are $X_i,K_{ij}$ for $i,j=1,\ldots,s$ and the last two lines are explicit. The last line is also equivalent to
\[
W_{n+1} = W_n + \sum_{i=1}^sb_i(X_i-J^{-1}X_j^\trans J+K_{ii}-J^{-1}K_{ii}^\trans J).
\]

\subsubsection*{3. $S=\GG^\bot$, $\GG\subseteq\SL(n,\Cc)$ $J$-quadratic}
\[
\left. \begin{array}{llll} 
X_i = - h(W_n+\sum_{j=1}^s a_{ij} X_j)B(\widetilde{W}_i)  \\
\\
K_{ij} = hB(\widetilde{W}_i) (\sum_{j'=1}^s (a_{ij'}X_{j'}+a_{jj'}K_{ij'}))  \\
\\
\widetilde{W}_i=W_n+\sum_{j=1}^s a_{ij} (X_j+J^{-1}X_j^\trans J+K_{ij})  \\
\\
W_{n+1} = W_n + h\sum_{i=1}^sb_i[B(\widetilde{W}_i) ,\widetilde{W}_i], 
\end{array}\right. 
\]
for $i,j=1,\ldots,s$, where the unknowns are $X_i,K_{ij}$ for $i,j=1,\ldots,s$ and the last two lines are explicit. The last line is also equivalent to
\[
W_{n+1} = W_n + \sum_{i=1}^sb_i(X_i+J^{-1}X_j^\trans J+K_{ii}+J^{-1}K_{ii}^\trans J).
\]

\proof[Proof of Theorem~\ref{thm_RK1}]
{\color{white}hej}
\begin{enumerate}
\item For $S:=\SL(n,\Cc)$ we have that $\NN(S)=\GL(n,\Cc)$ and $N(S)=GL(n,\Cc)$. 
Therefore the hypotheses of Theorem~\ref{thm1} are trivially satisfied.  
To get the explicit construction, we look at the argument of the gradient of the Hamiltonian which suggests to define
\begin{align*}
W_{n+1}&:= Q_{n+1}^\trans P_{n+1}\\
W_{n}&:= Q_{n}^\trans P_{n}\\
X_i&:=h Q_n^\trans K_i^P\\
Y_i&:=h(K_i^Q)^\trans P_n\\
K_{ij}&:= h^2\sum_{j'=1}^s a_{ij'}(K_j^Q)^\trans K_{j'}^P\\
\widetilde{W}_i&:=W_n+\sum_{j=1}^s a_{ij} (X_j+Y_j+K_{ij}),
\end{align*}
for $i,j=1,\ldots,s$. 
The equations for $X_i, Y_i$ are straightforward (consider the equations of the Runge--Kutta method \eqref{eq:symplecticRK} and take the transpose of the first equation and multiply by $P_n$, and multiply the second equation by $Q_n^\trans$, respectively).

To get the equations for $K_{ij}$, we first transpose the first equation of \eqref{eq:symplecticRK}, then we multiply it (indexed now by $j'$) by $h^2a_{ij'}K^P_{j'}$ and sum over $j'$. 
We thereby get
\[
K_{ij} = hB(\widetilde{W}_i)(\sum_{j'=1}^s (a_{ij'}X_{j'}+a_{ij'}\widetilde{K}_{jj'})) \qquad  \text{for } i,j=1,\ldots,s,
\]
where
\[
\widetilde{K}_{ij}:= h^2\sum_{j'=1}^s a_{ij'}(K_{j'}^Q)^\trans K_{j}^P.
\]
Multiplying the second equation of \eqref{eq:symplecticRK} (indexed now by $j'$) by $h^2a_{ij'}(K^Q_{j'})^\trans $ and then summing over $j'$, we obtain
\[
\widetilde{K}_{ij} = -h(\sum_{j'=1}^s (a_{ij'}Y_{j'}+a_{ij'}K_{jj'}))B(\widetilde{W}_i)   \qquad  \text{for } i,j=1,\ldots,s.
\]
Using then  
\[
\sum_{j'=1}^s\sum_{j''=1}^s a_{ij'}a_{jj''}(K_{j''}^Q)^\trans K_{j'}^P = \sum_{j'=1}^s\sum_{j''=1}^s a_{jj'}a_{ij''}(K_{j'}^Q)^\trans K_{j''}^P,
\]
for $i=1,\ldots,s$, we get
\[
\sum_{j'=1}^s a_{ij'}\widetilde{K}_{jj'} = \sum_{j'=1}^s a_{jj'} K_{ij'}  \qquad  \text{for } i,j=1,\ldots,s.
\]
Therefore the $\widetilde{K}_{ij}$ depend completely on the $K_{ij}$ and so we can neglect them, obtaining the desired equations for the $K_{ij}$.
Finally, to get the equation for $W_{n+1}$, we multiply the third one of (\ref{eq:symplecticRK}) transposed with the fourth one of (\ref{eq:symplecticRK})
and we get
\[
W_{n+1} = W_n + \sum_{i=1}^s b_i(X_i + Y_i) + h^2\sum_{i,j=1}^s b_ib_j(K^Q_i)^\trans  K^P_j.
\]
Using the symplecticity of the method, the last term becomes
\[
h^2\sum_{i,j=1}^s (b_ia_{ij}+b_ja_{ji})(K^Q_i)^\trans  K^P_j=\sum_{i=1}^s b_i(K_{ii}+\widetilde{K}_{ii}).
\]
Therefore,
\[
W_{n+1} = W_n + \sum_{i=1}^s b_i(X_i + Y_i + K_{ii}+\widetilde{K}_{ii}).
\]
Now substituting the equations found for $X_i,Y_i,K_{ii},\widetilde{K}_{ii}$ we get the desired equation for $W_{n+1}$.

\item
Symplectic Runge--Kutta methods preserve exactly the strong quadratic first integrals of a dynamical system. 
In particular, when $S$ is one of the spaces stated in the theorem, they preserve $N(S)=\lbrace Q\in GL(n,\Cc)|Q^\trans JQ=J\rbrace$. 
Therefore, by Theorem~\ref{thm1}, they descend to an integrator on $S$. 

It is also easy to check that, if we assume $B^\trans $ to be in $\NN(S)$, we get $Y_i=-J^{-1}X_i^\trans J$. Moreover, from the definition of $K_{ij}$ and $\widetilde{K}_{ij}$  and the equations:
\begin{align*}
K_{ij} &= hB(\widetilde{W}_i)(\sum_{j'=1}^s (a_{ij'}X_{j'}+a_{jj'}K_{ij'}))  \hspace{1cm}  \mbox{for } i,j=1,\ldots,s,\\
\widetilde{K}_{ij}&= -h(\sum_{j'=1}^s (a_{ij'}Y_{j'}+a_{ij'}K_{jj'}))B(\widetilde{W}_i)     \hspace{1cm}  \mbox{for } i,j=1,\ldots,s,
\end{align*} we get also that 
\[
-J^{-1}K_{ii}^\trans J=\widetilde{K}_{ii}.
\]
%
\endproof

\end{enumerate}

\begin{rem}
We stress that our methods are \emph{not} intrinsically formulated on~$S$.
That is, they depend on how $S$ is embedded as a subspace in $\GL(n,\Cc)$.
Therefore, there is no hope to present the schemes above only in terms of the matrix commutator.
\end{rem}

\begin{rem}
The order of convergence of the descended methods is the same as the underlying Runge--Kutta ones (see Figure~\ref{fig:rberr}), since if $Q_n=Q(nh)+\mathcal{O}(h^p)$ and $P_n=P(nh)+\mathcal{O}(h^p)$, then $W_n=Q_n^\trans P_n = W(nh) + \mathcal{O}(h^p)=Q(nh)^\trans P(nh) + \mathcal{O}(h^p)$.
\end{rem}

\begin{figure}
\begin{tikzpicture}
\node (img)  {\includegraphics[width=0.99\textwidth]{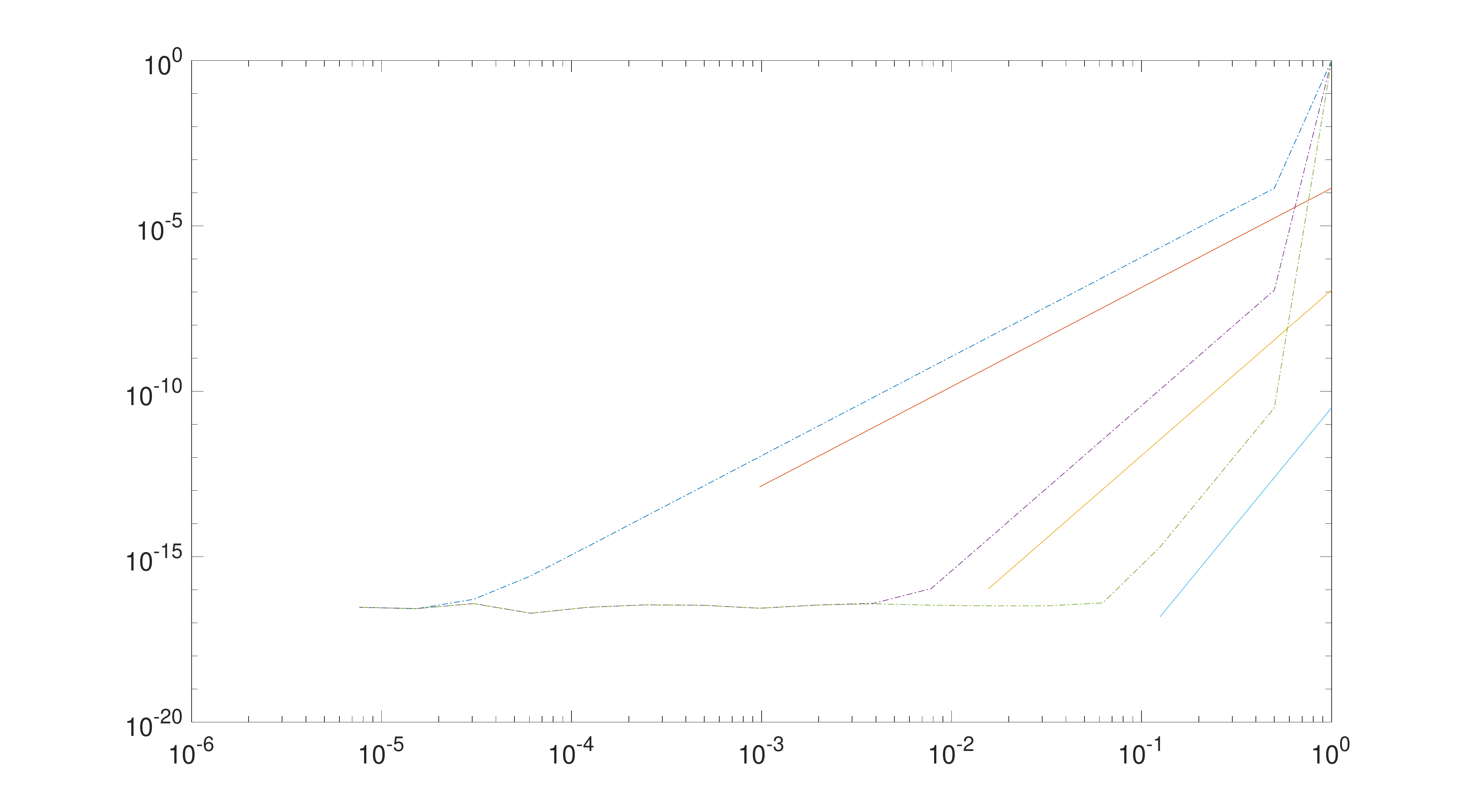}};
\centering
  \node[above=of img, yshift=-4em] {Error diagram for 2$^{nd}$, 4$^{th}$, 6$^{th}$ order schemes in Definition~\ref{def:isosyrk}};
 \node[below=of img, node distance=0cm, xshift=-6cm, yshift=4cm, rotate=90] {Error $\max_n\|W(nh)-W_n\|$};
  \node[below=of img,anchor=base, yshift=3em]{Time-step $h$};
 \end{tikzpicture}
\caption{Maximum error in total time $T=1s$, and time-step $h$, for $h=1, 0.5^{2},\dots,0.5^{17}$, in loglog scale, for 2$^{nd}$, 4$^{th}$, 6$^{th}$ order schemes in def~\ref{def:isosyrk}, respectively with dashed blue, purple and green line, applied to the generalized rigid body of section 5.1. The continuous lines are, respectively, red $h\mapsto h^3$, yellow $h\mapsto h^5$, blue $h\mapsto h^7$.}\label{fig:rberr}
\end{figure}
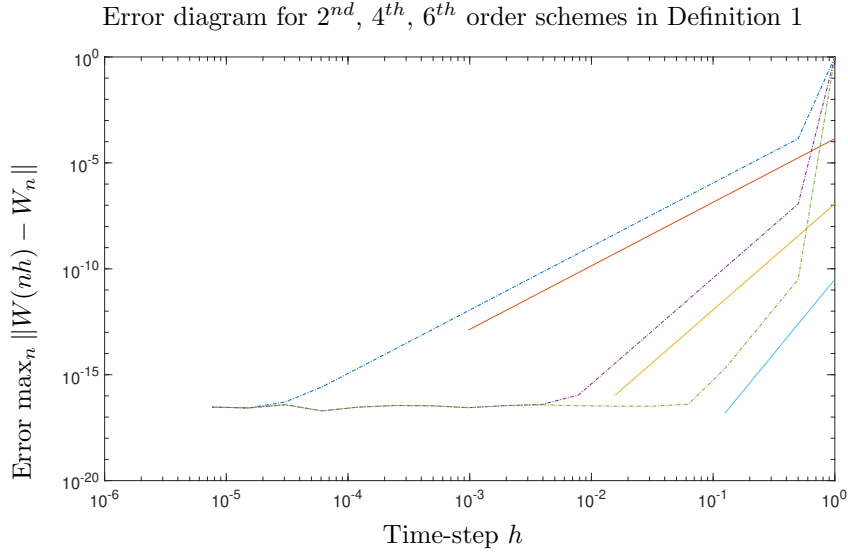

\subsection{Partitioned symplectic Runge--Kutta methods}

Given two Butcher tableaux
\[
\renewcommand\arraystretch{1.2}
\begin{array}
{c|ccc}
\widehat{c}_1  & \widehat{a}_{11} &\dots & \widehat{a}_{1s} \\
\vdots & \vdots &\ddots & \vdots \\
\widehat{c}_s &\widehat{a}_{s1} &\dots & \widehat{a}_{ss}\\
\hline
& \widehat{b}_1 &\dots &\widehat{b}_s 
\end{array}
\hspace{2cm}
\begin{array}
{c|ccc}
c_1  & a_{11} &\dots & a_{1s} \\
\vdots & \vdots &\ddots & \vdots \\
c_s &a_{s1} &\dots & a_{ss}\\
\hline
& b_1 &\dots &b_s 
\end{array}
\]
the associated partitioned Runge--Kutta method for (\ref{ham1}) is given by
\[
\left. \begin{array}{llll} 
K^Q_i = (Q_n+h\sum_{j=1}^s \widehat{a}_{ij} K^Q_j)B((Q_n+h\sum_{j=1}^s \widehat{a}_{ij} K^Q_j)^\trans (P_n+h\sum_{j=1}^s a_{ij} K^P_j))^\trans \\
\\
K^P_i = - (P_n+h\sum_{j=1}^s a_{ij} K^P_j)B((Q_n+h\sum_{j=1}^s \widehat{a}_{ij} K^Q_j)^\trans (P_n+h\sum_{j=1}^s a_{ij} K^P_j))   \\
\\ 
Q_{n+1} = Q_n + h\sum_{i=1}^s \widehat{b}_i K^Q_i\\
\\
P_{n+1} = P_n + h\sum_{i=1}^s b_i K^P_i,
\end{array}\right. 
\]
for $i,j=1,\ldots,s$. The partitioned Runge--Kutta method is \textit{symplectic}, i.e., the discrete flow is a symplectic map, if $b_i\widehat{a}_{ij}+\widehat{b}_ja_{ji}=b_i\widehat{b}_j$ and $\widehat{b}_i=b_i$ for $i,j=1,\ldots,s$.

\begin{thm}\label{thm_RK2}
Given two Butcher tableaux
\[
\renewcommand\arraystretch{1.2}
\begin{array}
{l|l}
\hat{\mathbf{c}}  & \hat{\mathbf{A}} \\
\hline
& \hat{\mathbf{b}}^\top  
\end{array}
\hspace{2cm}
\begin{array}
{l|l}
\mathbf{c}  & \mathbf{A} \\
\hline
& \mathbf{b}^\top 
\end{array}
\]
of a symplectic partitioned s-stage Runge--Kutta method, let $\Phi_h\colon T^*GL(n,\Cc) \to T^*GL(n,\Cc)$ denote the corresponding integrator map for the system~\eqref{ham1}.
Then:
\begin{enumerate}
\item The symplectic integrator $\Phi_h$ descends to a Lie--Poisson integrator $\phi_h$ on $\GL(n,\Cc)^*\simeq \GL(n,\Cc)$ for the isospectral Hamiltonian system~\eqref{eq:ham_isospectral}.
Furthermore, the map $\phi_h$ is completely constructive as an implicit integration scheme (see below for a specific formula).

\item If $S=\mathfrak{sl}(N,\Cc)$ is an invariant subspace of \eqref{eq:ham_isospectral} (as described above), then $\phi_h$ preserves $S$.

\item If $S=\GG$, and $S=\GG^\bot$, for $\GG$ a $J$-quadratic Lie subalgebra and $b_i\neq 0$, then $\phi_h$ preserves $S$ (for general Hamiltonians on $S$ extended to $\GL(n,\Cc)$) if and only if $a_{ij}=\widehat{a}_{ij}$, for $i,j=1,\ldots,s$.


 
\end{enumerate}
 
\end{thm}

The scheme obtained in Theorem~\ref{thm_RK2} is the following:

\subsubsection*{$S=\SL(n,\Cc)$ or $S=\GL(n,\Cc)$}
\[
\left. \begin{array}{llll} 
X_i = - h(W_n+\sum_{j=1}^s a_{ij} X_j)B(\widetilde{W}_i)   \\
\\
Y_i =  hB(\widetilde{W}_i) (W_n+\sum_{j=1}^s \widehat{a}_{ij} Y_j)  \\
\\
K_{ij} = hB(\widetilde{W}_i) (\sum_{j'=1}^s (a_{ij'}X_{j'}+\widehat{a}_{jj'}K_{ij'})) \\
\\
\widetilde{W}_i=W_n+\sum_{j=1}^s a_{ij} X_j+\widehat{a}_{ij} (Y_j+K_{ij}) \\
\\
W_{n+1} = W_n + h\sum_{i=1}^sb_i[B(\widetilde{W}_i) ,\widetilde{W}_i],
\end{array}\right. 
\]
for $i,j=1,\ldots,s$, where the unknowns are $X_i,Y_i,K_{ij}$ for $i,j=1,\ldots,s$ and the last two lines are explicit. 

\proof[Proof of Theorem~\ref{thm_RK2}]
{\color{white}hej}
\begin{enumerate}
\item The proof is, mutatis mutandis, identical to the one of the previous theorem. 
We have just to change accordingly the following definitions:

\begin{align*}
\widetilde{K}_{ij}&:= h^2\sum_{j'=1}^s \hat{a}_{ij'}(K_j^Q)^\trans K_{j'}^P\\
K_{ij}&:= h^2\sum_{j'=1}^s a_{ij'}(K_j^Q)^\trans K_{j'}^P\\
\widetilde{W}_i&:=W_n+\sum_{j=1}^s a_{ij} X_j+\widehat{a}_{ij} (Y_j+K_{ij}),
\end{align*}
and pointing out the following identity:
\[
\sum_{j'=1}^s a_{ij'}\widetilde{K}_{jj'} = \sum_{j'=1}^s \widehat{a}_{jj'} K_{ij'}  \hspace{1cm}  \mbox{for } i,j=1,\ldots,s.
\]
Finally we just use the condition of symplecticity for partitioned Runge--Kutta methods.

\item Follows directly from the formula for $W_{k+1}$.

\item  Partitioned symplectic Runge--Kutta methods preserve exactly the strong quadratic first integrals of a dynamical system if they are on the form $a(Q,P)$, where $a$ is a bilinear form on the space of matrices.
In particular, when $S$ is one of the spaces in statement (3) of the theorem, to preserve $N(S)=\lbrace Q\in GL(n,\Cc)|Q^\trans JQ=J\rbrace$ the method associated to $Q$-part has to preserve already the quadratic first integrals. This fact, together the condition of symplecticity of the partitioned Runge--Kutta methods, implies that $a_{ij}=\widehat{a}_{ij}$, for $i,j=1,\ldots,s$ whenever $b_i\neq 0$. 
\endproof
\end{enumerate} 

%
%
%

\subsection{Linear equivariance of the schemes}
In this paragraph we prove that the isospectral symplectic Runge--Kutta methods in Theorem~\ref{thm_RK1} and Theorem~\ref{thm_RK2} are linearly equivariant with respect to the invertible linear transformations that leave equations \eqref{iso_gen} and \eqref{eq:ham_isospectral} of the same form. 
Notice that the equations \eqref{iso_gen} and \eqref{eq:ham_isospectral} are not affine equivariant; linear equivariance is the best we can expect. 
The linear isomorphisms that leave equations \eqref{iso_gen} and \eqref{eq:ham_isospectral} invariant in form are, respectively, Lie algebra isomorphisms and Lie--Poisson isomorphisms. 
Indeed, consider a Lie algebra isomorphism $\mathcal A\colon \GL(n,\Cc)\to\GL(n,\Cc)$. 
Applying $\mathcal A$ to equation \eqref{iso_gen} gives
\[
\dfrac{d}{dt}(\mathcal {A} W) = \mathcal {A}[B(W),W] = [\mathcal {A}B(W),\mathcal {A}W]=[(\mathcal {A}\circ B\circ \mathcal {A}^{-1})(\mathcal {A}W),\mathcal {A}W],
\]
which shows the invariance in form of equation \eqref{iso_gen} to Lie algebra isomorphism. 
In particular, we have the identity
\begin{equation}\label{eq:Lie_mor}
\mathcal {A}[B(\mathcal {A}^{-1}W),\mathcal {A}^{-1}W] = [(\mathcal {A}\circ B\circ \mathcal {A}^{-1})(W),W].
\end{equation} 
Via the identification $\GL(n,\Cc)^*\simeq \GL(n,\Cc)$ as previously explained, it is easy to check that the adjoint operator $\mathcal A^*\colon \GL(n,\Cc)^*\to\GL(n,\Cc)^*$ acts on the coadjoint representation like $\mathcal {A}^*[X^\trans,Y] = [X^\trans(\mathcal {A}^*)^{-1},\mathcal {A}^*Y]$, for $X\in \GL(n,\Cc)$ and $Y\in\GL(n,\Cc)^*$. In particular, $\mathcal {A}^*$ is a Lie--Poisson map for equation \eqref{eq:ham_isospectral}:
\begin{align*}
\dfrac{d}{dt}(\mathcal {A}^* W) &= \mathcal {A}^*[\nabla H(W)^\trans,W] = [\nabla H(W)^\trans(\mathcal {A}^*)^{-1},\mathcal {A}^*W] \\ &= [\nabla (H\circ (\mathcal {A}^*)^{-1})^\trans(\mathcal {A}^*W)^\trans,\mathcal {A}^*W],
\end{align*}
which leaves equation \eqref{eq:ham_isospectral} invariant in form. 
We thus obtain the following identity:
\begin{equation}\label{eq:Lie--Poiss_mor}
\mathcal A^* [\nabla H((\mathcal A^*)^{-1}W)^\trans ,\mathcal (\mathcal A^*)^{-1}W] =  [\nabla( H\circ (\mathcal A^*)^{-1})(W)^\trans ,W].
\end{equation} 
For any map $B$ and Hamiltonian $H$, let $\phi_h(B)$ and $\phi_h(H)$, respectively, denote integrators as in Theorem~\ref{thm_RK1} or Theorem~\ref{thm_RK2}.
Now, the numerical scheme $\phi_h(B)$ is \emph{Lie equivariant}, and, correspondingly, $\phi_h(H)$ is \emph{Lie--Poisson equivariant} if
\begin{equation}\label{affine_equiv}
\mathcal A \circ \phi_h(B) = \phi_h(\mathcal A \circ B\circ\mathcal A^{-1}) \circ\mathcal A\\
\end{equation}
\begin{equation}\label{affine_equiv2}
\mathcal A^* \circ \phi_h(H)  = \phi_h(H\circ(\mathcal A^*)^{-1})\circ\mathcal A^*.
\end{equation}
The identities \eqref{eq:Lie_mor} and \eqref{eq:Lie--Poiss_mor} show that the right hand sides of equations \eqref{affine_equiv},\eqref{affine_equiv2} have the same form. Therefore, it is enough to prove equation \eqref{affine_equiv}. 


\begin{thm}\label{thm:affine_equiv}
	Let $\phi_h(B)$ be an isospectral (partitioned) symplectic Runge--Kutta method as in Theorem~\ref{thm_RK1} (or Theorem~\ref{thm_RK2}).
	Then $\phi_h(B)$ is Lie equivariant for any Lie morphism $\mathcal A\colon \GL(n,\Cc)\to\GL(n,\Cc)$.
\end{thm}

\begin{proof}
Let us consider equation \eqref{affine_equiv} for the partitioned symplectic Runge--Kutta schemes. 
The same conclusion for the symplectic Runge--Kutta method will follow straightforwardly from this.
We want to check equation \eqref{affine_equiv} for any $W_n\in\GL(n,\Cc)$ and $\mathcal A$ as above.
The right hand side is
\[
\left. \begin{array}{llll} 
X_i = - h\mathcal A(\mathcal A^{-1}(\mathcal AW_n+\sum_{j=1}^s a_{ij} X_j))B(\mathcal A^{-1}\widetilde{W}_i)    \\
\\
Y_i =  h\mathcal AB(\mathcal A^{-1}\widetilde{W}_i)  (\mathcal A^{-1}(AW_n+\sum_{j=1}^s \widehat{a}_{ij} Y_j)) \\
\\
K_{ij} = hAB(\mathcal A^{-1}\widetilde{W}_i)  (\sum_{j'=1}^s (\mathcal A^{-1}(a_{ij'}X_{j'}+\widehat{a}_{jj'}K_{ij'})))  \\
\\
\widetilde{W}_i=\mathcal AW_n+\sum_{j=1}^s a_{ij} X_j+\widehat{a}_{ij} (Y_j+K_{ij}) \\
\\
W_{n+1} = \mathcal AW_n + h\sum_{i=1}^sb_i\mathcal A[B(A^{-1}\widetilde{W}_i)  ,\mathcal A^{-1}\widetilde{W}_i], 
\end{array}\right. 
\]
for $i,j=1,\ldots,s$, which is equivalent to
\[
\left. \begin{array}{llll} 
\mathcal A^{-1}X_i = - h(W_n+\sum_{j=1}^s a_{ij} \mathcal A^{-1}X_j)B(\mathcal A^{-1}\widetilde{W}_i)   \\
\\
\mathcal A^{-1}Y_i =  hB(\mathcal A^{-1}\widetilde{W}_i)  (W_n+\sum_{j=1}^s \widehat{a}_{ij} \mathcal A^{-1}Y_j) \\
\\
\mathcal A^{-1}K_{ij} = hB(\mathcal A^{-1}\widetilde{W}_i)  (\sum_{j'=1}^s (a_{ij'}\mathcal A^{-1}X_{j'}+\widehat{a}_{jj'}\mathcal A^{-1}K_{ij'})) \\
\\
\mathcal A^{-1}\widetilde{W}_i=W_n+\sum_{j=1}^s a_{ij} \mathcal A^{-1}X_j+\widehat{a}_{ij} (\mathcal A^{-1}Y_j+\mathcal A^{-1}K_{ij})\\
\\
W_{n+1} = \mathcal A(W_n+ h\sum_{i=1}^sb_i[B(\mathcal A^{-1}\widetilde{W}_i)  ,\mathcal A^{-1}\widetilde{W}_i]),
\end{array}\right. 
\]
for $i,j=1,\ldots,s$. 
Relabeling $X_i:=\mathcal A^{-1}X_i,Y_i=\mathcal A^{-1}Y_i,K_{ij}:=A^{-1}K_{ij},\widetilde{W}_i:=\mathcal A^{-1}\widetilde{W}_i$ we get
\[
\left. \begin{array}{llll} 
X_i = - h(W_n+\sum_{j=1}^s a_{ij} X_j)B(\widetilde{W}_i)    \\
\\
Y_i =  hB(\widetilde{W}_i)  (W_n+\sum_{j=1}^s \widehat{a}_{ij} Y_j)  \\
\\
K_{ij} = hB(\widetilde{W}_i)  (\sum_{j'=1}^s (a_{ij'}X_{j'}+\widehat{a}_{jj'}K_{ij'}))  \\
\\
\widetilde{W}_i=W_n+\sum_{j=1}^s a_{ij} X_j+\widehat{a}_{ij} (Y_j+K_{ij}) \\
\\
W_{n+1} = \mathcal A(W_n + h\sum_{i=1}^sb_i[B(\widetilde{W}_i)  ,\widetilde{W}_i]) ,
\end{array}\right. 
\]
for $i,j=1,\ldots,s$ which is exactly the left hand side of (\ref{affine_equiv}).
\end{proof}

\section{Numerical examples} \label{sec:examples}

In this section we demonstrate the Isospectral Symplectic Runge--Kutta methods on some Hamiltonian isospectral flows often seen in the literature.
As expected, we obtain near conservation of the Hamiltonian (owing to the symplectic quality) and exact conservation (up to round-off errors) of the Casimir functions (owing to the isospectral quality).%
\footnote{The numerical experiments in this section are implemented in an easy-to-use MATLAB code, available at \href{https://bitbucket.org/Milo_Viviani/iso-runge-kutta}{\texttt{bitbucket.org/Milo\_Viviani/iso-runge-kutta}}.}

\subsection{The generalized rigid body} \label{sec:rigidbody}
The core example among Hamiltonian isospectral systems is the generalized rigid body. 
It is known that in any dimension $n$ it forms a complete integrable system in $\SO(n)$, as proved by Manakov~\cite{man}. 
The Hamiltonian is given by
\begin{equation}\label{eq:ham_rb}
H(W) = \frac{1}{2}\Tr((\mathcal{I}^{-1}W)^\trans W),\qquad W\in\SO(n),
\end{equation}
where $\mathcal{I}\colon\SO(n)\rightarrow\SO(n)$ is a symmetric positive definite inertia tensor.
The equations of motion are then
\begin{equation*}\label{rigid_body}
\left. \begin{array}{ll} 
	\dot{W} = -[\mathcal{I}^{-1}W,W] \\
	 W(0)=W_0.
			\end{array}\right.
\end{equation*} 

We discretize this system for $n=10$ with the method in Theorem~\ref{thm_RK1} and with the Butcher tableau corresponding to the implicit midpoint method.
Our implementation uses Newton iterations for the non-linear system.
The inertia tensor is given by 
\begin{equation*}
	(\mathcal{I}^{-1}W)_{ij} = \frac{W_{ij}}{i} , \quad i,j=1,\ldots,10
\end{equation*}
and we use the stepsize $h=0.1$.
The initial conditions are given by
\begin{equation*}
	(W_0)_{ij}=1/10 \quad\text{for}\quad i<j \qquad\text{and}\quad W_0^\trans = -W_0
\end{equation*}

As shown in Figure~\ref{fig:rb}, the Hamiltonian is nearly conserved and the Casimir functions are conserved up to the accuracy of the Newton iterations. 

\begin{figure}
\begin{minipage}{1\textwidth}
\begin{tikzpicture}
\centering
 \node (img)  {\includegraphics[scale=0.3]{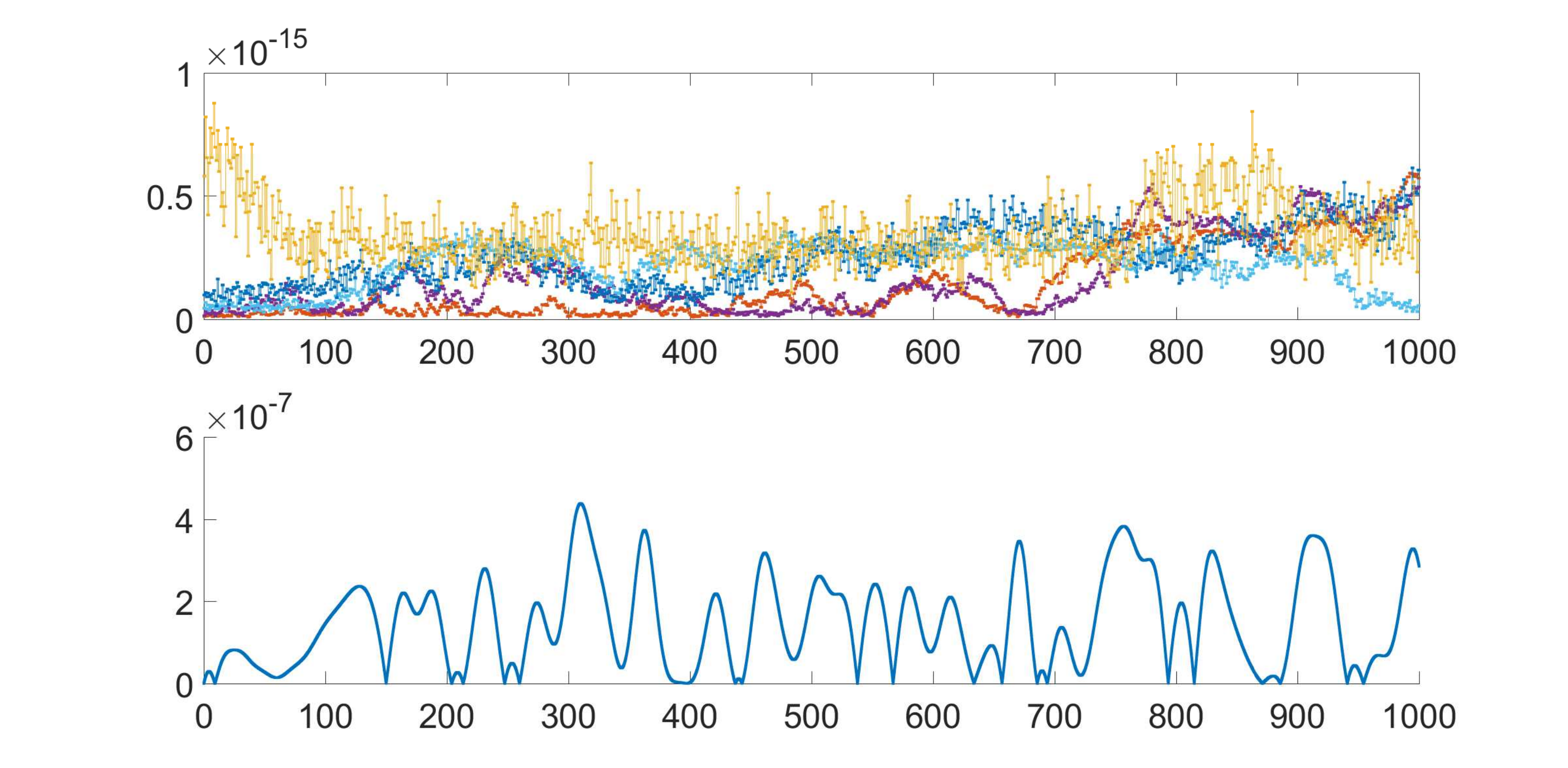}};

  \node[below=of img, node distance=0cm, yshift=4.1cm] {Evolution of error in Hamiltonian};
  \node[below=of img, node distance=0cm, yshift=7.2cm] {Evolution of error in Casimirs (eigenvalues)};
 \end{tikzpicture}
\end{minipage}

\caption{Evolution of errors for the generalized rigid body in $\SO(10)$. 
The Casimir functions correspond to the 10 eigenvalues (which occur in pairs). 
The Hamiltonian is given by \eqref{eq:ham_rb}.
The data for the simulation are given by: stepsize h=$0.1$; inertia tensor $\mathcal{I}=\mathrm{diag}(1,\ldots,10)$; initial conditions $(W_0)_{ij}=1/10$ if $i<j$, $(W_0)_{ij}=-1/10$ if $i>j$, $(W_0)_{ij}=0$ if $i=j$.}\label{fig:rb}
\end{figure}

The Casimirs of the generalized rigid body only constitutes $n$ first integrals, and they are therefore not enough to obtain the integrability.
The additional, non-Casimir first integrals are \emph{not} exactly preserved by our methods. 
However, from backward error analysis combined with KAM theory (see e.g.\ \cite{hlw}), one obtains that the additional integrals are nearly conserved (just as the Hamiltonian is nearly conserved).


\subsection{The (periodic) Toda lattice} \label{sec:todalattice}
Among Hamiltonian integrable systems the Toda lattice is perhaps the best known and most studied example. 
It represents a system of particles interacting pairwise with exponential forces. 
The equations of motion are determined by the Hamiltonian
\[
H(p,q) = \sum_{k=1}^n \left(\frac{1}{2}p_k^2+\exp(q_k-q_{k+1})\right),
\]
where $(q_i,p_i)$ are canonical coordinates of the $n$ particles. 
Independently, H\'enon~\cite{He1974}, Flaschka~\cite{Fl1974} and Manakov~\cite{man}
proved that the Toda system is integrable when $q_n=q_{n+1}$ (periodic boundary conditions). 
This is most easily seen by providing a Lax pair formulation. 
Indeed, by the following change of variables
\[
a_k =-\frac{1}{2}p_k, \hspace{2cm} b_k=\frac{1}{2}\exp\left(\frac{1}{2}(q_k-q_{k+1})\right),
\]
one obtains an equivalent isospectral flow
\begin{equation}\label{Toda}
\dot{L}=[B(L),L],
\end{equation}
where
\[
L=\left[\begin{matrix}
a_1 & b_1 & 0 & \ldots & b_n\\
b_1 & a_2 & b_2 & \ldots & 0\\
0 & b_2 & a_3 & \ldots & 0\\
\vdots & \vdots & \vdots & \ddots & \vdots  \\
b_n & 0 & 0 & \ldots & a_n
\end{matrix}\right],
\hspace{.5cm}
B(L)=\left[\begin{matrix}
0 & b_1 & 0 & \ldots & -b_n\\
-b_1 & 0 & b_2 & \ldots & 0\\
0 & -b_2 & 0 & \ldots & 0\\
\vdots & \vdots & \vdots & \ddots & \vdots \\
b_n & 0 & 0 & \ldots & 0
\end{matrix}\right].
\]
In these coordinates the canonical Hamiltonian is simply $H(L)=2\Tr(L^2)$.

So far, the mapping $B(\cdot)$ is defined only for matrices of the form $L$ above.
A natural extension
to any matrix $W\in\GL(n,\Cc)$ is
\begin{equation*}
	B(W)=
	\left[\begin{matrix}
0 & W_{12} & 0 & \ldots & -W_{1n}\\
-W_{21} & 0 & W_{23} & \ldots & 0\\
0 & -W_{32} & 0 & \ldots & 0\\
\vdots & \vdots & \vdots & \ddots & \vdots \\
W_{n1} & 0 & 0 & \ldots & 0
\end{matrix}\right].
\end{equation*}
Next, in order to extend \eqref{Toda} to a Hamiltonian isospectral flow on $\GL(n,\Cc)$ of the form in \eqref{eq:ham_isospectral}, we notice that we can take as a new Hamiltonian the function 
\begin{equation*}
	\widetilde{H}(W) = -\frac{1}{2}\Tr(W^\trans B(W)) + H(W) .	
\end{equation*}
The flow \eqref{eq:ham_isospectral} of this Hamiltonian then coincides with \eqref{Toda} for matrices of the form~$L$.
Indeed, since $B(W)\in \SO(n)$ when $W\in \Sym(n,\Rr)$, $\Tr(W^\trans B(W))=0$ for $W\in \Sym(n,\Rr)$.
Furthermore, $\nabla \widetilde{H}(W)^\trans=-B(W)$, when extended to any matrix $W$, since the linear mapping $B\colon\GL(n,\Cc)\to\GL(n,\Cc)$ is symmetric with respect to the Frobenius inner product.
Moreover, since the original Hamiltonian $H(W)$ is itself a Casimir function its gradient does not affect the dynamics.
We stress that the Hamiltonian structure of the extended system is \emph{different} from the original canonical Hamiltonian structure in the $q$ and $p$ variables.


We discretize the system for $n=4$ with the method in Theorem~\ref{thm_RK1} and with the Butcher tableau corresponding to the implicit midpoint method.
We use stepsize $h=0.1$ and initial conditions
\begin{equation*}
	a_i=b_i=(-1)^i, \qquad i=1,\ldots,4.
\end{equation*}
Since the $H(L)$ is one of the Casimir functions of the flow, it is preserved up to the iteration tolerance, as shown in Figure~\ref{fig:toda}.

We notice that in general our methods does not exactly preserve the zero entries of $L$ (although they are nearly preserved).
This is because the normalizer of the subspace of the symmetric matrices with the form of $L$ is not $J$-quadratic for $n>3$.


\begin{figure}
\begin{minipage}{1\textwidth}
\begin{tikzpicture}
 \node (img)  {\includegraphics[scale=0.3]{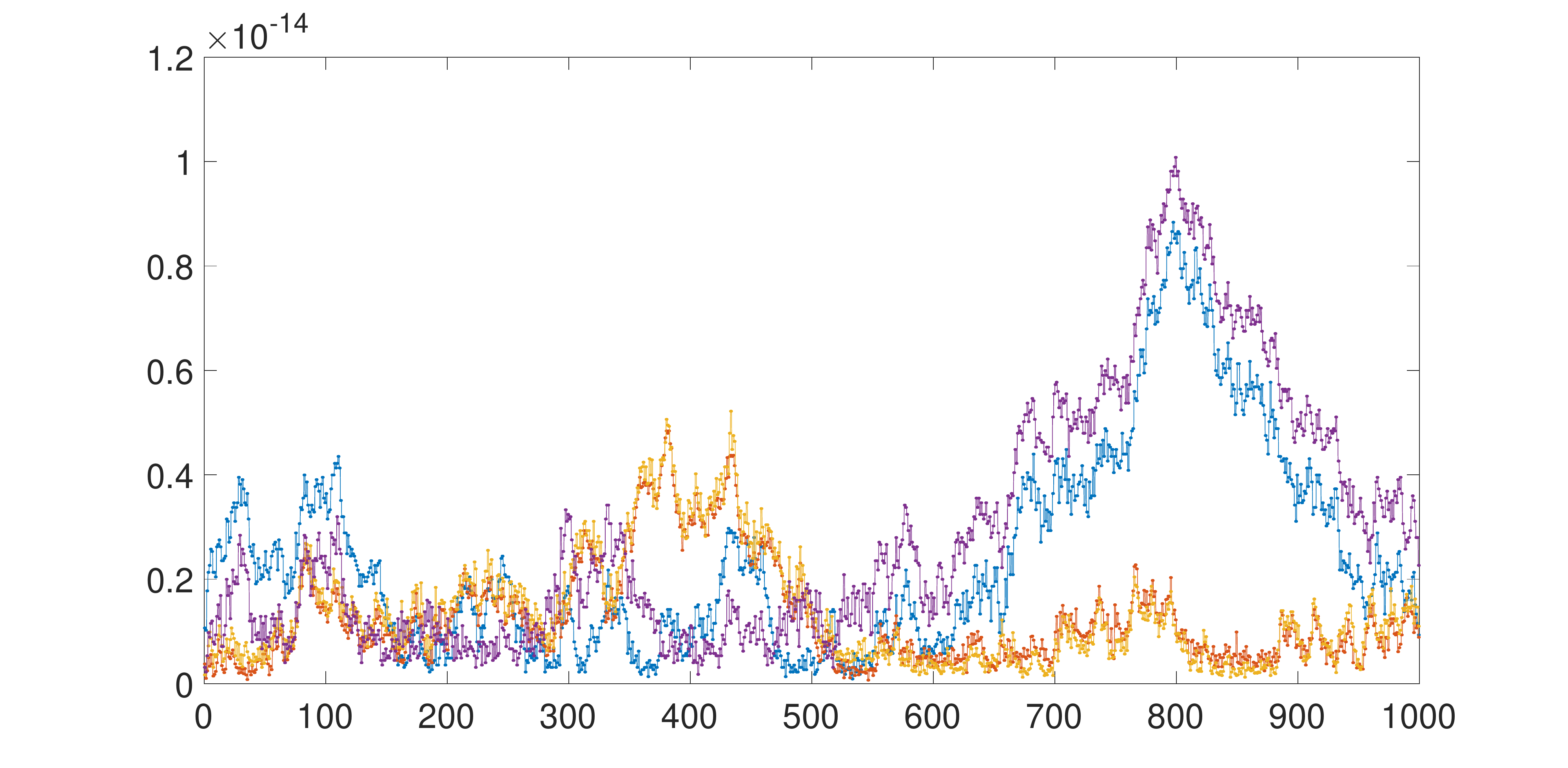}};
\centering
  \node[below=of img, node distance=0cm, yshift=7.2cm] {Evolution of error in Casimirs (eigenvalues)};
 \end{tikzpicture}
\end{minipage}

\caption{Error is Casimir functions for the periodic Toda Lattice with $n=4$. 
The data for the simulation are given by: stepsize $h=0.1$; initial conditions $a_i=b_i=(-1)^i$ for $i=1,\ldots,4$.}\label{fig:toda}
\end{figure}

\subsection{The Euler equations on a sphere} \label{sec:eulereq}
Let us briefly mention a beautiful approach for spatial discretization of the incompressible Euler equations on a sphere, which leads to a finite dimensional Hamiltonian isospectral flow.
For a full account we refer to the publication~\cite{MoVi2019}.\footnote{The Euler example is, in fact, the original motivation leading to the paper at hand.}

On the 2-sphere $\mathbb S^2$ the hydrodynamical Euler equations for an incompressible, inviscid, and homogeneous fluid can be formulated in terms of vorticity of the velocity vector.
The formulation is
\begin{equation}\label{euleq_vort1}
\left. \begin{array}{ll} 
	\dot{\omega} = \lbrace\Delta^{-1}\omega,\omega\rbrace \\
	 \omega(0)=\omega_0,
			\end{array}\right.
\end{equation} 
where the vorticity function $\omega$ is a smooth function on $\Ss^2$ with zero mean, $\Delta^{-1}$ is the inverse of the Laplace--Beltrami operator on the sphere (which is invertible because the kernel consists only of the constant functions), and $\lbrace\cdot,\cdot\rbrace$ is the Poisson bracket between functions.
This system is an infinite dimensional Lie--Poisson system on the dual of the algebra of divergence free vector field on $\Ss^2$.
The Casimir functions are given by
\begin{equation*}
	C(\omega) = \int_{\Ss^2} f(\omega(x))\,dx
\end{equation*}
where $f\colon\Rr\to\Rr$ is any smooth function.
Thus, there are infinitely many independent first integrals. 
(Although not enough integrals for the system to be integrable.)

In \emph{geometric quantization theory} (cf.~\cite{BaWe1997}) the space of smooth functions is replaced by a Hilbert space of linear operators, and the Poisson bracket $\{\cdot,\cdot\}$ is replaced by the commutator $[\cdot,\cdot]$.
The aim is to obtain a construction such that, in some sense, $[\cdot,\cdot]$ approximates $\{\cdot,\cdot\}$.
In his PhD thesis, Hoppe~\cite{hopPhD} gave an explicit quantization of $(C^\infty(\mathbb S^2),\{\cdot,\cdot\})$ in terms of the finite dimensional Lie algebras $\SU(n,\Cc)$, such that $[\cdot,\cdot] \to \{\cdot,\cdot\}$ as $n\to\infty$.
This naturally leads to a spatial discretization of the vorticity equation~\eqref{euleq_vort1} by simply replacing $\{\cdot,\cdot\}$ by $[\cdot,\cdot]$ and then working out what the corresponding discrete Laplacian $\Delta_n$ should be.
An explicit formula for $\Delta_n$ was given by Hoppe and Yau~\cite{hopyau}.
The resulting spatially discretized equations thus become


\begin{equation}\label{euleq_vort2}
\left. \begin{array}{ll} 
	\dot{W} = [\Delta^{-1}_n W ,W] \\
	 W(0)=W_0,
			\end{array}\right.
\end{equation} 
where $W\in\SU(n)$.
This is an isospectral Hamiltonian system with respect to the Hamiltonian $H(W)=\frac{1}{2}\Tr((\Delta^{-1}_n W)^\trans W)$. 

In our paper \cite{MoVi2019} we develop and further explore a fully discrete version of \eqref{euleq_vort2} based on the isospectral methods in this paper.
We thus obtain a discrete flow that preserves all the underlying structure of the Euler equations: conservation of Casimirs and the Lie--Poisson structure.
In particular, conservation of Casimirs is essential for numerical studies of the long-time behavior of \eqref{euleq_vort1} and of the mechanisms behind the \emph{inverse energy cascade} exclusive to 2D turbulence.



\subsection{Point vortices on a sphere and the Heisenberg spin chain} \label{sub:pointvortices}

As stated above, the symplectic isospectral Runge--Kutta methods are readily extended to product spaces.
For example we can deal with $(\SU(2)^*)^n$, where $n$ is the number of vortices or spin particles in the point-vortices equation~\cite{New2001} or, respectively, the Heisenberg spin chain~\cite{mmv}.

The Hamiltonian for point-vortex dynamics is
\[
H(W_1,W_2,\ldots,W_n)= -\frac{1}{4\pi}\sum_{\substack{i,j=1 \\ i<j}}^n \Gamma_i\Gamma_j \log\left(1-\dfrac{\Tr(W_i^\trans W_j)}{\|W_i\|^2\|W_j\|^2}\right),
\]
where $W_1,\ldots,W_n$ thought of a vectors in $\Rr^3$ are the positions of the point vortices and $\Gamma_1,\ldots,\Gamma_n$ the respective strengths.
For the Heisenberg spin chain the Hamiltonian is
\[
H(W_1,W_2,\ldots,W_n)= \sum_{i=1}^n \Tr(W_i^\trans W_{i+1}),
\]
where $W_1,\ldots,W_n$ are the spins of the particles and $W_1=W_{n+1}$, see for example~\cite{mmv}.

For these systems a new first integral arises, due to the $SU(2)$ symmetry of the Hamiltonians
\[
	H(GW_1G^{-1},GW_2G^{-1},\ldots,GW_nG^{-1})=H(W_1,W_2,\ldots,W_n),
\]
for any $G\in SU(2)$. 
The corresponding first integrals are given by the (weighted) sum of the vortices/spins
\[
M(W_1,W_2,\ldots,W_n)=\sum_{i=1}^n\Gamma_iW_i.
\]

We use the midpoint based numerical scheme of Theorem~\ref{thm_RK1} for the point-vortex Hamiltonian with $n=4$ and stepsize $h=0.1$.
The initial vortex positions are
\begin{equation*}
	x_1=[1 \ 0 \  0],x_2=[-1 \ 0 \ 0],x_3=[0 \ 1 \ 0],x_4=[0 \ -1 \ 0].
\end{equation*}
As before, the Casimirs are conserved and the Hamiltonian is nearly conserved.
In addition, the extra integral $M$ is conserved up to machine precision, as can be seen in Figure~\ref{fig:pointv}.

\begin{figure}
\begin{minipage}{1\textwidth}
\begin{tikzpicture}
 \node (img)  {\includegraphics[scale=0.3]{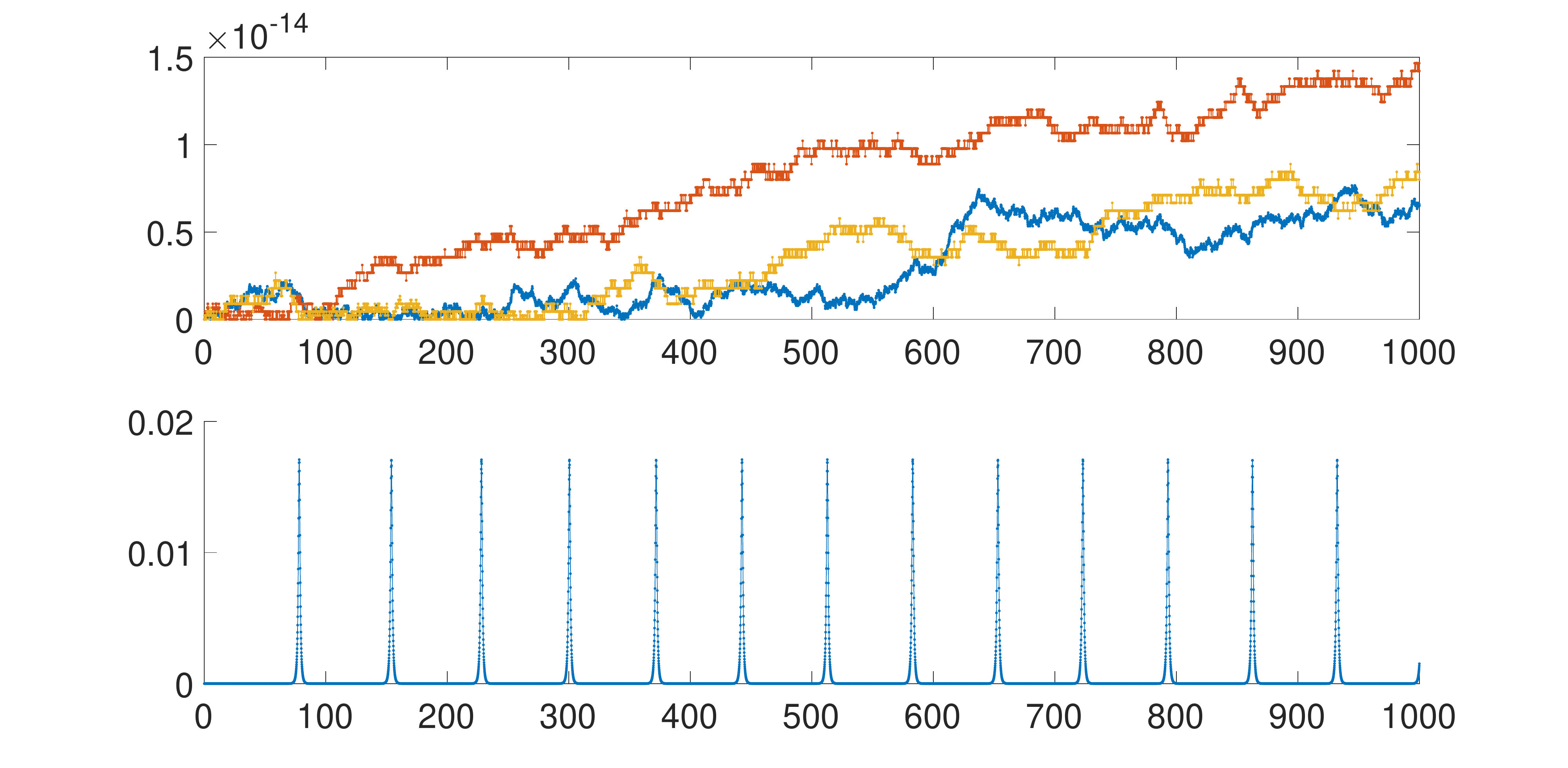}};
\centering
  \node[below=of img, node distance=0cm, yshift=4.2cm] {Hamiltonian variation};
  \node[below=of img, node distance=0cm, yshift=7.3cm] {Momentum variation};
 \end{tikzpicture}
\end{minipage}

\caption{Evolution of errors for 4 point vortices on a sphere. Upper: three compontents of the momentum $M$. Lower: Hamiltonian.
The data for the simulation are given by: stepsize $h=0.1$; initial conditions $x_1=[1 \ 0 \  0],x_2=[-1 \ 0 \ 0],x_3=[0 \ 1 \ 0],x_4=[0 \ -1 \ 0]$.}\label{fig:pointv}
\end{figure}

\subsection{The Bloch--Iserles flow}\label{sec:BloIse}
Given $N\in\SO(n)$, the Bloch--Iserles flow \cite{BloIse2006} on $\Sym(n,\Rr)$ is
\begin{equation*}
	\dot W = [W^2,N].
\end{equation*}
It can be cast as an isospectral flow \eqref{iso_gen} on $\Sym(n,\Rr)$ with $B(W) = NW+WN$.
Its interest lies in its integrable structure, which is fundamentally different from that of the Toda lattice and the generalized rigid body.

The Bloch--Iserles flow can be extended to a Hamiltonian isospectral flow on $\GL(n,\Rr)^*\simeq \GL(n,\Rr)$ such that $\Sym(n,\Rr)$ is an invariant subspace, just as the Toda flow in Section~\ref{sec:todalattice} above.
The Hamiltonian for this is 
\[
H(W) = \Tr(W^2 N).
\]


We give a numerical example with $n=3$ and, again, the second order midpoint based scheme of Theorem~\ref{thm:main_isosyrk} with stepsize $h=0.1$. 
The matrix $N$ and the initial conditions are
\[
N=\frac{1}{\sqrt{2}}\begin{bmatrix}
    0 & 1 & 0  \\
    -1 & 0 & 1 \\
    0 & -1 & 0
\end{bmatrix} \qquad\text{and}\qquad
W_0=
\begin{bmatrix}
    
    0.0163 &   0.3928 &   0.2415 \\
    0.3928 &   0.1501  &  0.3443 \\
    0.2415 &   0.3443  &  0.6603
\end{bmatrix}.
\]
The evolution of energy and the Casimirs (eigenvalues) are given in \autoref{fig:BI1}.
The integrable structure of the flow is revealed as quasi-periodicity in projections of the phase diagram, as seen in \autoref{fig:BI2}.


\begin{figure}[h!]
\begin{minipage}{1\textwidth}
\begin{tikzpicture}
 \node (img)  {\includegraphics[scale=0.3]{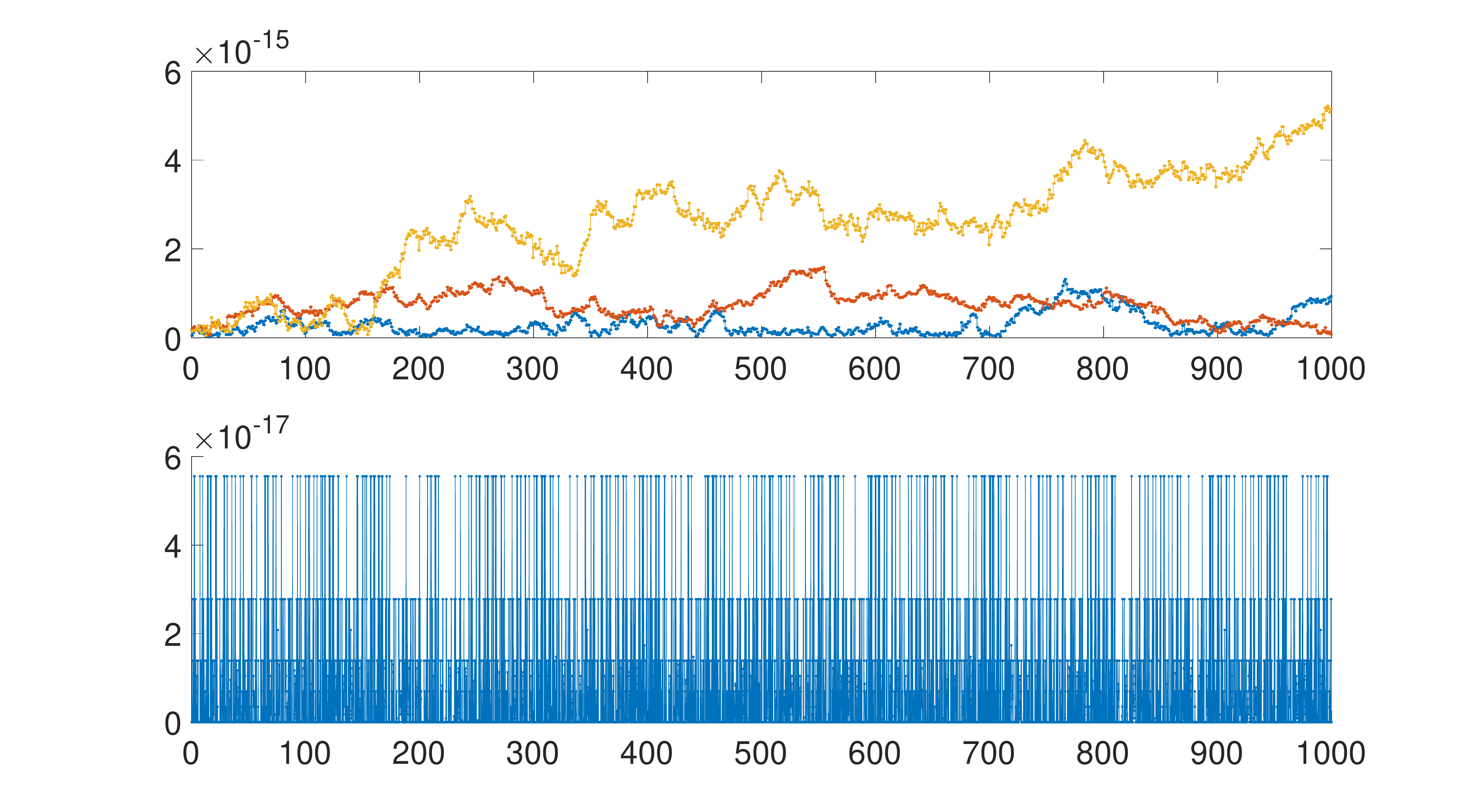}};
\centering

  \node[below=of img, node distance=0cm, yshift=4.5cm] {Evolution of error in Hamiltonian};
  \node[below=of img, node distance=0cm, yshift=7.8cm] {Evolution of error in Casimirs (eigenvalues)};
 \end{tikzpicture}
\end{minipage}
\caption{Casimir and Hamiltonian variation in time $T=100$, for the Bloch--Iserles flow in $\mbox{Sym}(3)$ and time-step $h=0.1$.}\label{fig:BI1}
\end{figure}

\begin{figure}[h!]
\begin{tikzpicture}
 \node (img) at (0,0)  {\includegraphics[width=0.8\textwidth]{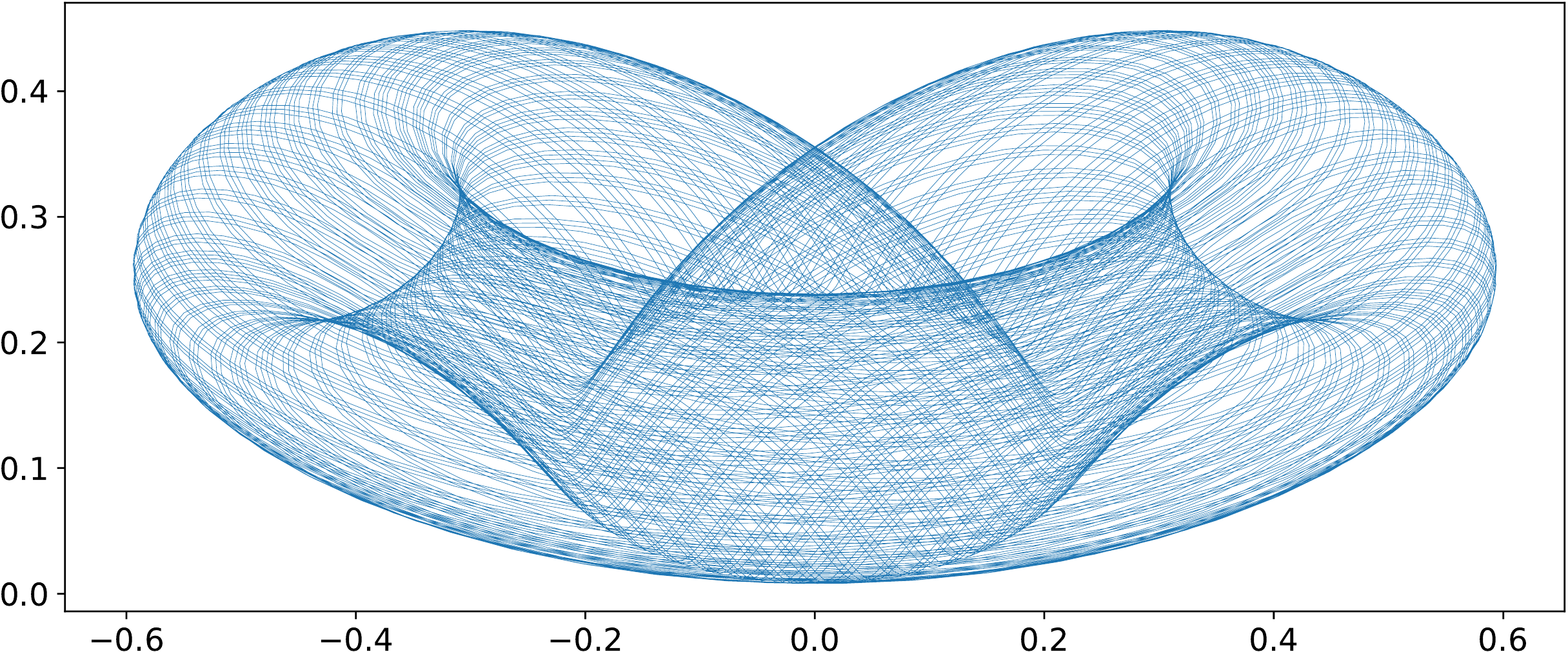}};
  \node[above=of img,yshift=-7ex] {Phase space portrait projected to the $(W_{12},W_{13})$-plane};
 \end{tikzpicture}
\caption{Projected phase space portrait for the Bloch--Iserles system. The resulting diagram reveals quasi-periodic motion on embedded tori as expected for integrable systems.}\label{fig:BI2}
\end{figure}

\subsection{The Toeplitz inverse eigenvalue problem}\label{sec:toeplitz}

In this section we demonstrate that the methods introduced can be applied also to non-Hamiltonian systems.
To this end, consider Chu's flow on symmetric real matrices, which is of the form~\eqref{iso_gen} with
\[
B(W)=\left[\begin{matrix}
0 & W_{1,1}-W_{2,2} & W_{1,2}-W_{2,3} & \ldots & W_{1,n-1}-W_{2,n}\\
W_{2,2}-W_{1,1} & 0 & W_{2,2}-W_{3,3} & \ldots & W_{2,n-1}-W_{3,n} \\
W_{3,2}-W_{2,1} & W_{3,3}-W_{2,2} & 0 & \ldots & W_{3,n-1}-W_{4,n} \\
\vdots & \vdots & \vdots & \vdots & \vdots \\
W_{n,2}-W_{n-1,1} & W_{n,3}-W_{n-1,2} &  W_{n,4}-W_{n-1,3} & \ldots & 0 \\
\end{matrix}\right]
\]
Notice that if $W\in \Sym(n,\Rr)$ then $B(W)\in\mathfrak{so}(n)$. 

The Toeplitz inverse eigenvalue problem reads as follows.
Given a certain set of eigenvalues, find a symmetric Toeplitz matrix with that prescribed spectra (recall that a Toeplitz matrix is a matrix with constant elements on the diagonals).
In \cite{Lan1994}, H.J. Landau established that, for any given spectra, there exists a symmetric Toeplitz matrix with those eigenvalues.
Towards a practical algorithm, Chu~\cite{Ch1994} instead proved that fixed points of the isospectral flow with $B(W)$ as above are symmetric Toeplitz matrices, provided the eigenvalues are distinct.

Chu's flow is particularly interesting from a numerical point of view because there exist periodic orbits. 
Thus, the flow does not always converge to a fixed point.
However, the periodic orbits are unstable and because of the floating point drift
in numerical methods Chu's flow in practice always converge to a symmetric Toeplitz matrix when the starting point has distinct eigenvalues \cite{webbphd}. 

A qualitatively better simulation of Chu's flow, which preserves the periodic orbits, can be obtained by restriction to centrosymmetric matrices. 
A matrix is said to be centrosymmetric if it is invariant with respect to a rotation of the components of $\pi$ grade.
In other words a matrix $A$ is centrosymmetric if
\[
A E-EA=0,
\]
where
 \[E=\left[\begin{matrix}
0 & 0 & \ldots & 0 & 1\\
0 & 0 & \ldots & 1 & 0\\
\vdots & \vdots & \ddots & \vdots & \vdots \\
0 & 1 & \ldots & 0 & 0\\
1 & 0 & \ldots & 0 & 0
\end{matrix}\right].
\]

The set of centrosymmetric matrices of dimension $n$ is a Lie algebra which we denote by $\mbox{Centro}(n)$.
In particular we have that the symmetric Toeplitz matrices are centrosymmetric and $B(W)$ is centrosymmetric when $W$ is symmetric and centrosymmetric \cite{webbphd}. 
Therefore, for the Toeplitz problem, Chu's flow can be restricted to the symmetric-centrosymmetric matrices. 
With this restriction the periodic orbits are numerically preserved, 
and therefore the simulation of the flow is more realistic. In fact, in order to avoid the drifting out from the periodic orbits, it is necessary to respect the centrosymmetric symmetry of the original flow in the discrete approximation.

If $S$ denotes the linear subspace of symmetric-centrosymmetric matrices, then by Theorem~\ref{thm1} the isospectral symplectic Runge--Kutta methods descend to an isospectral integrator on the symmetric-centrosymmetric matrices, provided that $B(W)$ is in the normalizer of $S$, which is $\SO(n)\cap\mbox{Centro}(n)$.

We use the midpoint based numerical scheme of Theorem~\ref{thm_RK1} for Chu's flow with $n=4$ and stepsize $h=0.1$.
The initial conditions, proposed in \cite{webbphd}, are
 \[
 W_0=\left[\begin{matrix}
0.1336 & 0 & 0 & 0.5669\\
0 & -0.1336 & 0.378 & 0\\
0 & 0.378 & -0.1336 & 0 \\
0.5669 & 0 & 0 & 0.1336
\end{matrix}\right]
\]
Figures~\ref{fig:chu1}--\ref{fig:chu2} show the difference of the behaviour of the flow with and without the restriction to the centrosymmetric matrices, confirming the same predictions presented in \cite{webbphd}.
\begin{figure}
\begin{minipage}{1\textwidth}
\begin{tikzpicture}
 \node (img)  {\includegraphics[scale=0.3]{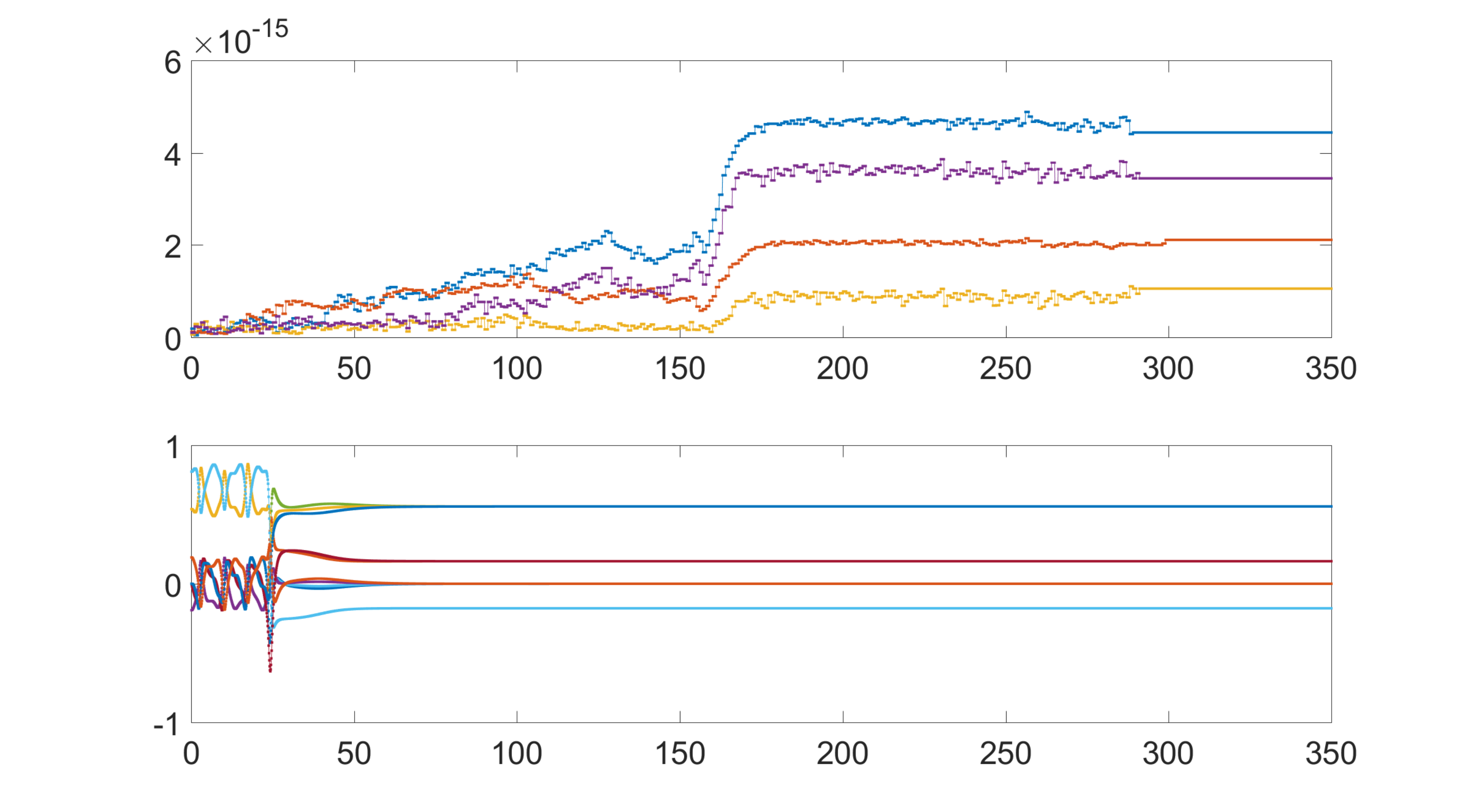}};
\centering
  \node[below=of img, node distance=0cm, yshift=4.7cm] {Evolution of components} ;
  \node[below=of img, node distance=0cm, yshift=7.9cm]{Evolution of error in Casimirs (eigenvalues)};
 \end{tikzpicture}
\end{minipage}
\caption{Numerical simulation of Chu's flow without forcing centrosymmetry of $B(W)$.}\label{fig:chu1}
\end{figure}

\begin{figure}
\begin{minipage}{1\textwidth}
\begin{tikzpicture}
 \node (img)  {\includegraphics[scale=0.3]{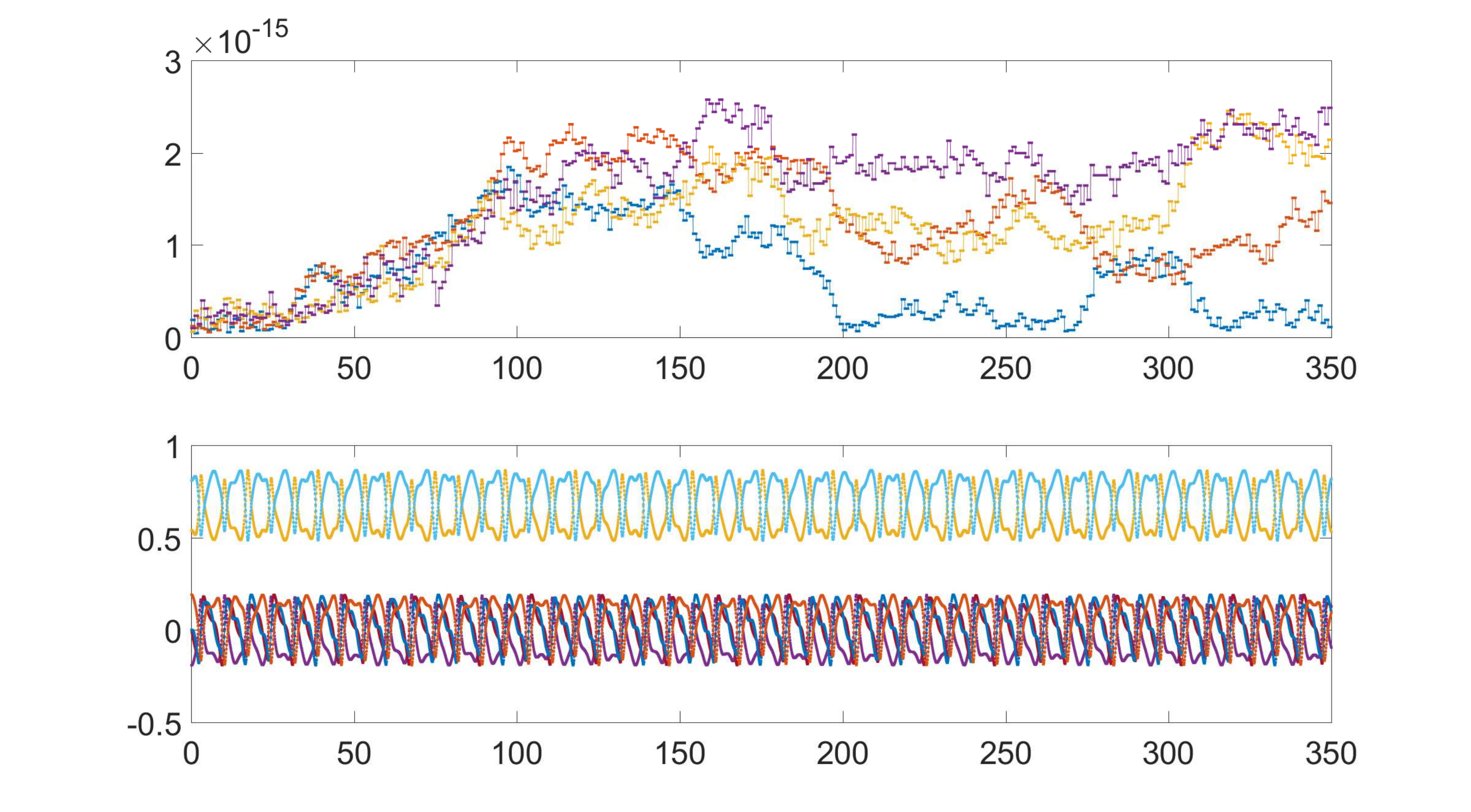}};
\centering
  \node[below=of img, node distance=0cm, yshift=4.7cm] {Evolution of components} ;
  \node[below=of img, node distance=0cm, yshift=7.9cm]{Evolution of error in Casimirs (eigenvalues)};
 \end{tikzpicture}
\end{minipage}
\caption{Numerical simulation of Chu's flow forcing centrosymmetry of $B(W)$.}\label{fig:chu2}
\end{figure}

\subsection{The Brockett flow}
Another example of a non-Hamiltonian isospectral flow is the \textit{Brockett flow}, or \textit{double bracket flow}
\begin{equation}\label{eq:Brockett}
\dot{W} = [[N,W],W],
\end{equation}
where $N$ and $W$ are $n\times n$ self-adjoint complex matrices.
In \cite{Bro1991}, Brockett shows that for a diagonal $N$  with distinct entries and $W_0$ a self-adjoint matrix with distinct eigenvalues, $W(t)$ converges exponentially fast to a diagonal matrix with the eigenvalues sorted accordingly to the order of the entries of $N$. 
There are interesting connections between the Brockett flow and information theory.
Indeed, the Brockett flow can be viewed as a gradient flow, with respect to the Fisher--Rao information metric, of a relative entropy functional on the statistical manifold of multivariate Gaussian distributions~\cite[Sec.~3.4.3]{Mo2017}.

We apply the isospectral midpoint method with $h=0.1$.	
In Figure~\ref{fig:eig_Brockett} we plot the eigenvalues and the components variation for a randomly generated self-adjoint initial matrix $W_0$ of dimension $3\times 3$ and $N=\operatorname{diag}(1,2,3)$. 
Figure~\ref{fig:eig_Brockett} displays the exponential convergence to a similar diagonal matrix.

\begin{figure}[h!]
\begin{minipage}{1\textwidth}
\begin{tikzpicture}
 \node (img)  {\includegraphics[scale=0.3]{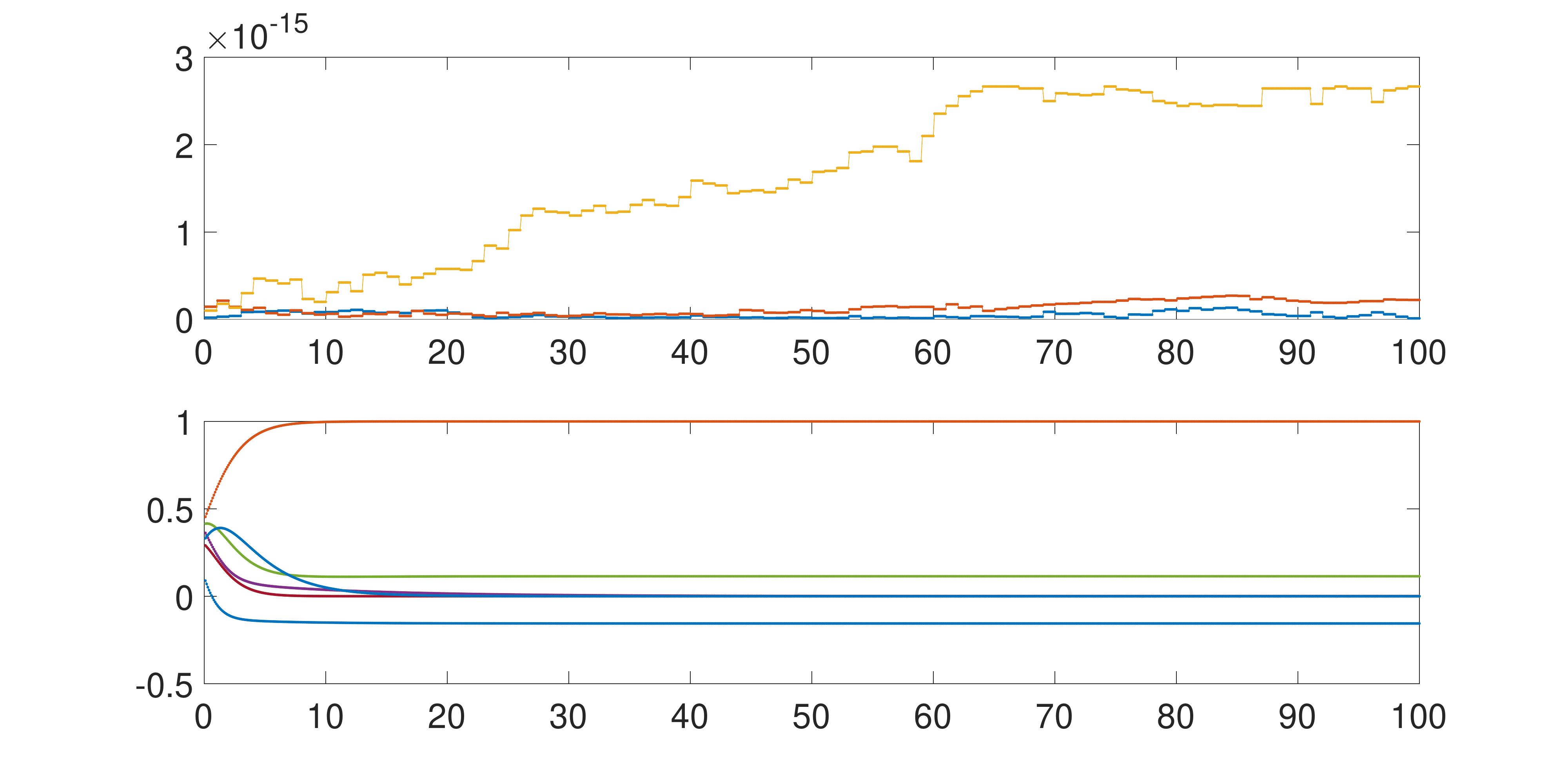}};
\centering
   \node[below=of img, node distance=0cm, yshift=7.4cm] {Casimir (eigenvalues) variation in time};
    \node[below=of img, node distance=0cm, yshift=4.3cm] {Evolution of components} ;
 \end{tikzpicture}
\end{minipage}
\caption{Eigenvalues (above) and components (below) for the Brockett flow \eqref{eq:Brockett} solved with the midpoint IsoSyRK method, with time-step $h=0.1$. 
The initial condition $W_0$ is a randomly generated self-adjoint matrix of dimension $3\times 3$ and $N=\operatorname{diag}(1,2,3)$.}\label{fig:eig_Brockett}
\end{figure}

\bibliographystyle{amsplain}
\bibliography{biblio2}

 \end{document}